\documentclass{amsart}

\usepackage{amsmath,amssymb,enumerate,amsthm,graphicx}
\usepackage[initials]{amsrefs}
 
\numberwithin{equation}{section}

\theoremstyle{plain}
\newtheorem{theorem}{Theorem}[section]
\newtheorem{proposition}[theorem]{Proposition}
\newtheorem{lemma}[theorem]{Lemma}
\newtheorem{corollary}[theorem]{Corollary}

\theoremstyle{definition}
\newtheorem{remark}[theorem]{Remark}
\newtheorem{example}[theorem]{Example}

\newcommand{\UP}{\blacktriangle}                
\newcommand{\DOWN}{\blacktriangledown}          

\begin{document}

\title{Rough sets determined by tolerances}

\author[J.~J{\"a}rvinen]{Jouni J{\"a}rvinen}
\address{J.~J{\"a}rvinen, Sirkankuja 1, 20810~Turku, Finland}
\email{Jouni.Kalervo.Jarvinen@gmail.com}
\urladdr{\url{http://sites.google.com/site/jounikalervojarvinen/}}

\author[S.~Radeleczki]{S{\'a}ndor Radeleczki}
\thanks{Acknowledgements: The research of the second author was carried out as
part of the TAMOP-4.2.1.B-10/2/KONV-2010-0001 project supported by the
European Union, co-financed by the European Social Fund.}
\address{S.~Radeleczki, Institute of Mathematics\\ 
University of Miskolc\\3515~Miskolc-Egyetemv{\'a}ros\\Hungary}
\email{matradi@uni-miskolc.hu}
\urladdr{\url{http://www.uni-miskolc.hu/~matradi/}}

\begin{abstract}
We show that for any tolerance $R$ on $U$, the ordered sets of lower and upper rough approximations 
determined by $R$ form ortholattices. These ortholattices are completely distributive, thus forming 
atomistic Boolean lattices, if and only if  $R$ is induced by an irredundant covering of $U$, and in 
such a case, the atoms of these Boolean lattices are described.
We prove that the ordered set $\mathit{RS}$ of rough sets determined by a tolerance $R$ on $U$
is a complete lattice if and only if it is a complete subdirect product of the complete lattices 
of lower and upper rough approximations. We show that $R$ is a tolerance 
induced by an irredundant covering of $U$ if and only if $\mathit{RS}$ is an algebraic completely
distributive lattice, and in such a situation a quasi-Nelson algebra can be defined  on $\mathit{RS}$. 
We present necessary and sufficient conditions which guarantee that for a tolerance $R$ on $U$, 
the ordered set $\mathit{RS}_X$ is a lattice for all  $X \subseteq U$, where $R_X$ denotes the 
restriction of $R$ to the set $X$ and $\mathit{RS}_X$ is the corresponding set of rough sets.
We introduce the disjoint representation and the formal concept representation of rough
sets, and show that they are Dedekind--MacNeille completions of $\mathit{RS}$. 
\end{abstract}


\maketitle

\section{Introduction}

Rough sets were introduced in \cite{Pawl82} by Z.~Pawlak. The
key idea is that our knowledge about the properties of the objects of a 
given universe of discourse $U$ may be inadequate or incomplete in the sense that the
objects of the universe $U$ can be observed only within the accuracy of indiscernibility relations.
According to Pawlak's original definition, an indiscernibility relation $E$ on $U$ 
is an equivalence relation interpreted so that two elements of $U$ are $E$-related if they cannot 
be distinguished by their properties known by us. Thus, indiscernibility relations allow us to partition 
a set of objects into classes of indistinguishable objects. 
For any subset $X \subseteq U$, the \emph{lower approximation} $X^\DOWN$
of $X$ consists of  elements such that their $E$-class is
included in $X$, and the \emph{upper approximation} $X^\UP$ of $X$
is the set of the elements whose $E$-class intersects with $X$. This means that
$X^\DOWN$ can be viewed as the set of elements certainly belonging 
to $X$, because all elements $E$-related to them are also in $X$.
Similarly, $X^\UP$ may be interpreted as the set of elements that
possibly are in $X$, because in $X$ there is at least one
element indiscernible to them. The \emph{rough set} of $X$ 
is the pair $(X^\DOWN,X^\UP)$ and the set of all 
rough sets is
\begin{equation*}
 \mathit{RS} = \{ (X^\DOWN, X^\UP)  \mid X \subseteq U \}. 
\end{equation*}
The set $\mathit{RS}$  may be canonically ordered by the
coordinatewise order: 
\begin{equation*}
(X^\DOWN,X^\UP) \leq (Y^\DOWN,Y^\UP) \iff X^\DOWN \subseteq Y^\DOWN \mbox{ \
and \ } X^\UP \subseteq Y^\UP. 
\end{equation*}

In \cite{PomPom88} it was proved that $\mathit{RS}$ is a lattice 
which forms also a Stone algebra.  Later this result was improved in 
\cite{Com93} by showing that $\mathit{RS}$ is in fact a regular double Stone algebra. 
Therefore, $\mathit{RS}$ determines also a three-valued {\L}ukasiewicz algebra 
and a semi-simple Nelson algebra, because it is well known that these three types of algebras
can be transformed to each other \cite{Pagliani97}. 

In the literature can be found numerous generalizations of rough sets such that 
equivalences are replaced by relations of different types. For instance, it is
known that in the case of quasiorders (reflexive and transitive binary relations),
a Nelson algebra such that  the underlying rough set lattice is an 
algebraic lattice can be defined on $\mathit{RS}$  \cite{JarRad,JRV09}.
If rough sets are determined by relations that are symmetric and transitive, 
then the structure of $\mathit{RS}$ is analogous to the case of equivalences \cite{Jarv04}.
For a more general approach in the case of partial equivalences, see \cite{Mani08}.
There exist also studies in which approximation operators are defined in terms
of an arbitrary binary relation -- this idea was first proposed in \cite{YaoLin1996}.
In \cite{Dzik2013}, expansions of bounded distributive lattices equipped with a Galois connection
are represented in terms of rough approximation operators defined by arbitrary binary relations. One may
also observe that in the current literature new approximation operators based on different
viewpoints are constantly being proposed (see e.g. \cite{AbuDonia2012,Ma2012} for some recent studies).

In this paper, we assume that indiscernibility relations are tolerances 
(reflexive and symmetric binary relations). The term \emph{tolerance relation} was introduced
in the context of visual perception theory by E.~C.~Zeeman \cite{Zeeman62}, 
motivated by the fact that indistinguishability of ``points'' in the visual 
world is limited by the discreteness of retinal receptors.
One can argue that tolerances suit better for representing indistinguishability
than equivalences, because transitivity is the least obvious property of indiscernibility. 
Namely, we may have a finite sequence of objects $x_1, x_2, \ldots, x_n$ such that each 
two consecutive objects $x_i$ and $x_{i+1}$ are indiscernible, but there is a notable difference
between $x_1$ and $x_n$. It is known  \cite{Jarv99,Jarv01} that in the case of tolerances, 
$\mathit{RS}$ is not necessarily a lattice if the cardinality of $U$ is greater than four. 
Our main goals in this work are to find conditions under which $\mathit{RS}$ forms a lattice, 
and, in case $\mathit{RS}$ is a lattice, to study its properties.
 
As mentioned, originally rough set approximations were defined in terms of equivalences,
being bijectively related to partitions. In this paper, we consider tolerances, which
are closely connected to coverings. In the literature can be found several ways to define approximations in terms of coverings
(see recent surveys in \cite{Restrepo2013,Yao2012}), and in this work we connect our approximation operators 
to some covering-based approximation operators, also.

The paper is organized as follows: In Section~\ref{Sec:Preliminaries}, we
present the definition of rough approximation operators and present their essential 
properties. In addition, we give preliminaries of Galois connections, ortholattices, and formal
concepts. Section~\ref{Sec:ToleranceApproximations}
is devoted to the rough set operators defined by tolerance relations.
Starting from the well-known fact that for any tolerance on $U$, the pair $({^\UP},{^\DOWN})$
is a Galois connection on the power set lattice of $U$ and characterize
rough set approximation pairs as certain kind of Galois connections
$(F,G)$ on a power set. We show that 
$\wp(U)^\DOWN = \{ X^\DOWN \mid X \subseteq U \}$
and $\wp(U)^\UP = \{ X^\UP \mid X \subseteq U \}$ form
ortholattices and prove that these ortholattices are completely distributive if and only if
$R$ is induced by an irredundant covering of $U$. Note that distributive
ortholattices are Boolean lattices, and a Boolean lattice is atomistic
if and only if it is completely distributive. This means that
$\wp(U)^\DOWN$ and $\wp(U)^\UP$ are atomistic Boolean lattices exactly when $R$ 
is induced by an irredundant covering of $U$, and we describe the atoms of these lattices. 
In Section~\ref{Sec:OrderedSets}, we study the ordered set of rough sets 
$\mathit{RS}$ and show that it can be up to isomorphism identified with a set of pairs 
$ \{ (\mathcal{I}(X),\mathcal{C}(X)) \mid X \subseteq U \}$, where
$\mathcal{I}$ and $\mathcal{C}$ are interior and closure operators
on the set $U$ satisfying certain conditions.  We prove that $\mathit{RS}$ is a
complete lattice if and only if it is a complete subdirect
product of $\wp(U)^\DOWN$ and $\wp(U)^\UP$. We also show that $\mathit{RS}$ is an 
algebraic completely distributive lattice if and only if $R$ is induced by an irredundant 
covering of $U$, and in such a case, on $\mathit{RS}$ a quasi-Nelson algebra can be defined. 
The section ends with necessary and sufficient conditions which guarantee that
for a tolerance $R$ on $U$, the ordered set $\mathit{RS}_X$ is a lattice for all 
$X \subseteq U$, where $R_X$ denotes the restriction of $R$ to the set 
$X$ and $\mathit{RS}_X$ is the set of all rough sets determined by $R_X$. Finally,
Section~\ref{Sec:DisjointRepresentations} is devoted to the disjoint representation 
and the formal concept representation of rough sets. In particular, we prove that
these representations are Dedekind--MacNeille completions of $\mathit{RS}$.

\section{Preliminaries: Rough approximation operators, Galois connections, and formal concepts}
\label{Sec:Preliminaries}

First we recall from \cite{Jarv07} some notation and basic properties of rough approximation operators
defined by arbitrary binary relations. Let $R$ be a binary relation on the set $U$. 
For any $X \subseteq U$,
we denote $$R(X) = \{ y \in U \mid x \, R \,y \text{ for some $x \in X$} \}.$$
For the singleton sets, $R(\{x\})$ is written simply as $R(x)$, that is,
$R(x) = \{ y \in U \mid x \, R \, y\}$. It is clear that 
$R(X) = \bigcup_{x\in X} R(x)$ for all $X \subseteq U$. 
The \emph{lower approximation} of a set $X \subseteq U$ is 
\[
 X^\DOWN = \{ x \mid R(x) \subseteq X\}
\]
and
$X$'s \emph{upper approximation} is 
\[
 X^\UP = \{ x \mid R(x) \cap X \neq \emptyset \}.
\]

Let $\wp(U)$ denote the \emph{power set} of $U$. It is a complete Boolean lattice
with respect to the set-inclusion order. The
map $^\UP$ is a complete join-homomorphism on $\wp(U)$, that is, it
preserves all unions:
\[
  \big ( \bigcup_{X \in \mathcal{H}} X \big )^\UP = \bigcup_{X \in \mathcal{H}} X^\UP.
\]
Analogously, $^\DOWN$ is a complete meet-homomorphism on $\wp(U)$ preserving all intersections:
\[
\big ( \bigcap_{X \in \mathcal{H}} X \big )^\DOWN = \bigcap_{X \in \mathcal{H}} X^\DOWN.
\]
Hence, the approximation operators are order-preserving, that is, $X \subseteq Y$ implies
$X^\DOWN \subseteq Y^\DOWN$ and $X^\UP \subseteq Y^\UP$.
In addition, approximation operators are dual, meaning that for all $X \subseteq U$, $X^{c\UP} = X^{\DOWN c}$
and $X^{c\DOWN} = X^{\UP c}$, where $X^c$ denotes the set-theoretical \emph{complement} $U \setminus X$ of $X$.
By the above, the set
\[
  \wp(U)^\DOWN = \{ X^\DOWN \mid X \subseteq U\}
\]
is a \emph{closure system}, that is, it is closed under arbitrary intersections. Similarly,
\[ 
\wp(U)^\UP = \{ X^\UP \mid X \subseteq U\} \
\]
forms an interior system, that is, it is closed under any union.
The complete lattices $ \wp(U)^\DOWN$ and $\wp(U)^\UP$ are with respect 
to the set-inclusion relation dually 
order-isomorphic by the map $X^\DOWN \mapsto X^{\DOWN c} = X^{c \UP}$.

For two ordered sets $P$ and $Q$, a pair 
$(f,g)$ of maps $f \colon P \to Q$ and $g \colon Q \to P$ 
is called a  \emph{Galois connection\/}  between $P$ and $Q$
if for all $p \in P$ and $q \in Q$,
\[
f(p) \leq q \iff p \leq g(q).
\]
In the next lemma are listed some of the known properties of Galois connections.

\begin{lemma}\label{Lem:GaloisProperty}
Let $(f,g)$ be a Galois connection between two ordered sets $P$ and $Q$.
\begin{enumerate}[\rm (a)] 
\item The composition $f \circ g \circ f$ equals $f$ and the composition 
$g \circ f \circ g$ equals $g$.

\item The composition $g \circ f$ is a lattice-theoretical closure operator on $P$ and the set
  of $g \circ f$-closed elements is $g(Q)$, that is, $(g \circ f)(P) = g(Q)$.

\item The composition $f \circ g$ is a lattice-theoretical interior operator on $Q$ and 
  the set of $f \circ g$-closed elements is $f(P)$, that is,
 that is, $(f \circ g)(Q) = f(P)$.

\item The image sets $f(P)$ and $g(Q)$ are order-isomorphic.

\item The map $f$ is a complete join-homomorphism and $g$ is a complete meet-homomorphism.

\item The maps $f$ and $g$ uniquely determine each other by the equations
     \[ f(p) = \bigwedge \{ q \in Q \mid p \leq g(q)\} \mbox{ \ and \ }
     g(q) = \bigvee   \{ p \in P \mid f(p) \leq q\} .\] 
\end{enumerate}
\end{lemma}

In addition, if the maps $f \colon P \to Q$ and $g \colon Q \to P$ between two complete
lattices $P$ and $Q$ form a Galois connection $(f,g)$, then $f(P)$ is a complete lattice such that for all
$S \subseteq f(P)$,
\[ \textstyle
\bigvee S = \bigvee_Q S 
\ \mbox{ and } \
\bigwedge S = f \big ( g \big (\bigwedge_Q S \big ) \big) = 
f \big ( \bigwedge_P g \big ( S \big ) \big),
\]
and $g(Q)$ is a complete lattice such that for all $S \subseteq g(Q)$,
\[ \textstyle
\bigvee S = g \big ( f \big (\bigvee_P S \big ) \big) = 
g \big ( \bigvee_Q f \big ( S \big ) \big)
\ \mbox{ and } \
\bigwedge S = \bigwedge_P S.
\]

An \emph{orthocomplementation} on a bounded lattice is a function that maps each element 
$x$ to an \emph{orthocomplement} $x^\bot$ in such a way that the following axioms hold:
\begin{enumerate}[({O}1)]
 \item $x \leq y$ implies $y^\bot \leq x^\bot$;
 \item $x^{\bot \bot} = x$;
 \item $x \vee x^\bot = 1$ and $x \wedge x^\bot = 0$.
\end{enumerate}
An \emph{ortholattice} is a bounded lattice equipped with an orthocomplementation. 
Ortholattices are self-dual by the map $^\bot$. Note that if an ortholattice
is distributive, then it is a Boolean lattice such that
the complement of the element $x$ is  $x^\bot$.

Let $(f,g)$ be a Galois connection on a Boolean lattice $(B,\vee,\wedge,{^c},0,1)$.
We may define the maps $^\bot \colon f(B) \to f(B)$ and $^\top \colon g(B) \to g(B)$
by setting 
\[
f(x)^\bot = f(f(x)^c) \text{ \ and \ } g(x)^{\top} = g(g(x)^c).
\]
The maps $^\bot$ and $^\top$ satisfy (O1), and if $f$ and $g$ are dual, that is, 
$f(x)^c = g(x^c)$ for all $x \in B$, then $^\bot$ and $^\top$ satisfy (O2).
Additionally, if $f$ is extensive, that is, $x \leq f(x)$ for all $x \in B$, then $^\bot$
and $^\top$ satisfy (O3). These observations are summarized in the following
well-known lemma.

\begin{lemma}\label{Lem:Ortho}
If $(f,g)$ is a Galois connection on a Boolean lattice $B$ such
that $f$ and $g$ are dual, and  $f$ is extensive, then $f(B)$ and $g(B)$
are ortholattices.
\end{lemma}

We end this section by presenting some terminology concerning formal concepts from 
\cite{ganter1999formal}. A \emph{formal context} $\mathbb{K} = (G,M,I)$ consists of two
sets $G$ and $M$ and a relation $I$ from $G$ to $M$. The elements of $G$ are 
called the \emph{objects} and the elements of $M$ are called \emph{attributes} of
the context. We write $g \, I \, m$\, or\, $(g,m) \in I$ to mean that the object $g$ has the
attribute $m$. For $A \subseteq G$ and $B \subseteq M$, we define
\[ A' = \{ m \in M \mid g \, I \, m \text{ for all } g \in A\} 
\text{ \ and \ } B' = \{ g \in G \mid g \, I \, m \text{ for all } m \in B\}.
\]
A \emph{formal concept} of the context $(G,M,I)$ is a pair $(A,B)$ with
$A \subseteq G$, $B \subseteq M$, $A' = B$, and $B' = A$. 
We call $A$ the \emph{extent} and $B$ the \emph{intent} of the
concept $(A,B)$. It is easy to see that $(A,B) \in \wp(G) \times \wp(M)$
is a concept if and only if $(A,B) = (A'',A') = (B',B'')$.
The set of all concepts of the context $\mathbb{K} = (G,M,I)$ is denoted 
by $\mathfrak{B}(\mathbb{K})$. 
The set $\mathfrak{B}(\mathbb{K})$ is ordered by
\begin{equation} \label{Eq:ConceptOrder}
 (A_1,B_1) \leq (A_2,B_2) \iff A_1 \subseteq A_2 \iff B_1 \supseteq B_2.
\end{equation}
With respect to this order, $\mathfrak{B}(\mathbb{K})$ forms a complete lattice,
called the \emph{concept lattice} of the context $\mathbb{K}$, in which 
\[
 \bigvee_{j \in J} (A_j,B_j) = \Big ( \big ( \bigcup_{j \in J} A_j \big )'', \bigcap_{j \in J}B_j \Big ) 
\text{ \ and \ }
 \bigwedge_{j \in J} (A_j,B_j) = \Big ( \bigcap_{j \in J}A_j , \big ( \bigcup_{j \in J} B_j \big )'' \Big ).
\]

\section{Approximation operations defined by tolerances}
\label{Sec:ToleranceApproximations}

In this section, we are recalling from \cite{Jarv99,Jarv07} the characteristic properties of 
rough sets approximation operators defined by \emph{tolerance relations}, which are reflexive 
and symmetric binary  relations. Also some new results are presented. It is known that the relation
$R$ is reflexive if and only if $X^\DOWN \subseteq X \subseteq X^\UP$ for all $X \subseteq U$.
Similarly, $R$ is symmetric if and only if the pair $({^\UP}, {^\DOWN})$ is a Galois
connection on $\wp(U)$. In the rest of this section, we assume that $R$ is a tolerance on 
$U$. Note that for all $X \subseteq U$, we have $X^\UP = R(X) = \bigcup_{x \in X} R(x)$. 

\begin{proposition}\label{Prop:GaloisCharacterization}
Let $(F,G)$ be a Galois connection on the complete lattice $\wp(U)$. Then,
there exists a tolerance $R$ on $U$ such that $F$ equals ${^\UP}$ and
$G$ equals ${^\DOWN}$ if and only if the following conditions hold for
all $x,y \in U$:
\begin{enumerate}[\rm (i)]
 \item $x \in F(\{x\})$;
 \item $x \in F(\{y\})$ implies $y \in F(\{x\})$.
\end{enumerate}
\end{proposition}

\begin{proof}
($\Rightarrow$)\, Suppose that $F$ equals ${^\UP}$ and $G$ equals ${^\DOWN}$ for some tolerance $R$.
Then, condition (i) means that $x \in R(x)$ for all $x \in X$ and (ii) is equivalent
to that $x \in R(y)$ implies $y \in R(x)$ for any $x,y \in U$. These conditions are obviously 
satisfied, because $R$ is a tolerance.

($\Leftarrow$)\, Let us define a binary relation $R$ by setting $x \, R \, y$ if and only
if $x \in F(\{y\})$. Because $F$ satisfies (i) and (ii), and  $R(x) = F(\{x\})$ for all $x \in U$, 
the relation $R$ is a tolerance. In addition,
\[ X^\UP = \bigcup_{x \in X} R(x) = \bigcup_{x \in X} F(\{x\}) = F \big ( \bigcup_{x \in X} \{x\} \big )
         = F(X),
\]
because $F$ is a complete join-homomorphism. Since the pairs of maps forming Galois connections are
unique by Lemma~\ref{Lem:GaloisProperty}(f), we have that $G$ must equal $^\DOWN$.
\end{proof}

Because  $({^\UP}, {^\DOWN})$ is a Galois connection on $\wp(U)$, the approximation operators 
have all the properties listed in Lemma~\ref{Lem:GaloisProperty}. In particular, the map 
$X \mapsto X^{\UP\DOWN}$
is the closure operator corresponding to the closure system $\wp(U)^\DOWN$,
which forms a complete lattice with respect to the order $\subseteq$ such that
\begin{equation}\label{Eq:DownLattice}
\bigvee_{X \in \mathcal{H}} X^\DOWN = \big ( \bigcup_{X \in \mathcal{H}} X^\DOWN \big )^{\UP\DOWN}
\quad \mbox{ and  } \quad 
\bigwedge_{X \in \mathcal{H}} X^\DOWN = \bigcap_{X \in \mathcal{H}} X^\DOWN  
\end{equation}
for all $\mathcal{H} \subseteq \wp(U)$. In addition,
$\wp(U)^\DOWN = \{ X^{\UP\DOWN} \mid X \subseteq U \}$.

Analogously, the map $X \mapsto X^{\DOWN\UP}$
is the interior operator that corresponds the interior system $\wp(U)^\UP$, 
which is a complete lattice such that
\begin{equation}\label{Eq:UpLattice}
\bigvee_{X \in \mathcal{H}} X^\UP = \bigcup_{X \in \mathcal{H}} X^\UP 
\quad \mbox{ and  } \quad 
\bigwedge_{X \in \mathcal{H}} X^\UP 
=  \big (\bigcap_{X \in \mathcal{H}} X^\UP \big )^{\DOWN \UP}
\end{equation}
for all $\mathcal{H} \subseteq \wp(U)$. The closure system $\wp(U)^\UP$ can be
also written in the form
$\{ X^{\DOWN\UP} \mid X \subseteq U \}$. 

By Lemma~\ref{Lem:Ortho}, $\wp(U)^\UP$ and $\wp(U)^\DOWN$ are ortholattices.
In $\wp(U)^\UP$, the orthocomplementation is $^\bot \colon X^\UP \mapsto X^{\UP c \UP}$,
and the map $^\top \colon X^\DOWN \mapsto X^{\DOWN c \DOWN}$ is the
orthocomplementation operation of $\wp(U)^\DOWN$. \label{Def:Ortho}
Hence, $\wp(U)^\UP$ and $\wp(U)^\DOWN$ are self-dual, and 
\[ (\wp(U)^\UP,\subseteq) \cong (\wp(U)^\UP,\supseteq) \cong(\wp(U)^\DOWN,\subseteq) \cong(\wp(U)^\DOWN,\supseteq). \]

Next we study the relationship between the lattices of approximations 
and concept lattices. For a tolerance $R$ on a set $U$, we consider the context 
$\mathbb{K} = (U,U,R^c)$, whose concept lattice is
$\mathfrak{B}(\mathbb{K}) = \{ (X'',X') \mid X \subseteq U\}$.
For $X \subseteq U$,
\begin{align*} 
X' = \{ x \in U \mid y \, R^c \, x \text{ for all } y \in X\} 
   = \{ x \in U \mid (x,y) \notin R \text{ for all } y \in X\} 
    = X^{\UP c}.
\end{align*}
Thus, $X^\UP = X^{\prime \, c}$ and $X^\DOWN = X^{c \UP c} = X^{c \, \prime}$.
In addition, $X'' = X^{\UP\DOWN}$ and hence
\[ 
  \mathfrak{B}(\mathbb{K}) = \{ (X^{\UP\DOWN},X^{c \DOWN}) \mid X \subseteq U\}.
\]
If $(A,B) \in \mathfrak{B}(\mathbb{K})$, then $A,B \in \wp(U)^\DOWN$ such that $A = B'$ and
$B = A'$. For any $A \in \wp(U)^\DOWN$, $A' = A^{\UP c} = A^{c \DOWN} = A^\top$, 
where $^\top$ is the orthocomplement defined in $\wp(U)^\DOWN$.
Thus, $(A,B) = (A,A^\top) = (B^\top,B)$. On the other hand, if $A \in \wp(U)^\DOWN$,
then $A = A^{\UP \DOWN}$ and $(A,A^\top)$ belongs to $\in \mathfrak{B}(\mathbb{K})$. Hence,
\[ 
  \mathfrak{B}(\mathbb{K}) = \{ (A,A^\top) \mid A \in \wp(U)^\DOWN \}.
\]
Notice that in the literature can be found studies in which notions of formal concept
analysis are applied to rough set theory.
Particularly,  Y.~Y.~Yao considers in \cite{yao2004concept} so-called 
``complement contexts'', which actually lead us to study the contexts of the
form $(U,U,R^c)$.

Let $\wp(U)^{\DOWN \mathrm{op}}$ denote the dual of the lattice $\wp(U)^{\DOWN}$, that is, 
$(\wp(U)^\DOWN,\supseteq)$.

\begin{proposition} \label{Prop:ConceptIsom}
Let $R$ be a tolerance on a set $U$ and\/ $\mathbb{K} = (U,U,R^c)$.
\begin{enumerate}[\rm (a)]
\item The complete lattices $\wp(U)^\UP$, $\wp(U)^\DOWN$, and $\mathfrak{B}(\mathbb{K})$ are isomorphic.
\item The concept lattice $\mathfrak{B}(\mathbb{K})$ is a complete sublattice of 
 $\wp(U)^\DOWN \times \wp(U)^{\DOWN \mathrm{op}}$.
\end{enumerate}
\end{proposition}

\begin{proof}
(a) It is obvious that the map $A \mapsto (A,A^\top)$ is an isomorphism between
$\wp(U)^\DOWN$ and $\mathfrak{B}(\mathbb{K})$, and we have already
noted that  $\wp(U)^\DOWN$ and  $\wp(U)^\UP$ are isomorphic.

(b) Clearly, $\mathfrak{B}(\mathbb{K}) \subseteq \wp(U)^\DOWN \times \wp(U)^\DOWN$. 
For all $\{A_j\}_{j \in J} \subseteq \wp(U)^\DOWN$, the join in $\wp(U)^{\DOWN}$
is $\bigvee_{j \in J} A_j = ( \bigcup_{j \in J} A_j )^{\UP\DOWN} =
 ( \bigcup_{j \in J} A_j )''$, and the meet in $\wp(U)^\DOWN$
is $\bigwedge_{j \in J} A_j = \bigcap_{j \in J} A_j$. Thus, the join in $\wp(U)^{\DOWN \mathrm{op}}$
is $\bigvee_{j \in J} A_j = \bigcap_{j \in J} A_j$. Therefore, for any
$\{(A_j,B_j)\}_{j \in J} \subseteq \mathfrak{B}(\mathbb{K})$, its join $\bigvee_{j \in J} (A_j,B_j)$
coincides in  $\mathfrak{B}(\mathbb{K})$ and $\wp(U)^\DOWN \times \wp(U)^{\DOWN \mathrm{op}}$.
An analogous observation can be done with respect to meets. Thus, $\mathfrak{B}(\mathbb{K})$
is a complete sublattice of $\wp(U)^\DOWN \times \wp(U)^{\DOWN \mathrm{op}}$.
\end{proof}

Note that Proposition~\ref{Prop:ConceptIsom} implies that for a tolerance $R$ and for the 
context $\mathbb{K} = (U,U,R^c)$, the concept lattice $\mathfrak{B}(\mathbb{K})$ is an 
ortholattice and the orthocomplement is obtained by swapping the sets (cf. \cite[p.~54]{ganter1999formal}),
in other words, for $(A,A^\top) \in \mathfrak{B}(\mathbb{K})$, its orthocomplement is $(A^\top, A)$. 

Now we may present a characterization of complete ortholattices in terms of rough sets
operators defined by tolerances.

\begin{proposition} \label{Prop:OrthoCharacterization}
A complete lattice $L$ forms an ortholattice if and only if there exists a set $U$
and a tolerance $R$ on $U$ such that $L \cong \wp(U)^\DOWN \cong \wp(U)^\UP$.
\end{proposition}

\begin{proof}
 As noted, for a tolerance $R$ on $U$, the complete lattices  $\wp(U)^\DOWN \cong \wp(U)^\UP$
 are ortholattices. Conversely, it is known that if $L$ is a complete ortholattice, then
 there exists a context $\mathbb{K} = (U,U,I)$, where $I$ is an irreflexive and symmetric binary
 relation on $U$, such that $L \cong \mathfrak{B}(\mathbb{K})$; see
 \cite{ganter1999formal}. The maps $A \mapsto A'$ and $B \mapsto B'$ defined in this
 context form an order-reversing Galois connection between $(\wp(U),\subseteq)$ and $(\wp(U),\supseteq)$.
 Because $X \mapsto X^c$ is an order-isomorphism between  $(\wp(U),\supseteq)$ and $(\wp(U),\subseteq)$,
 the composite maps
 \[ 
 f \colon A \mapsto A^{\prime c} \qquad \text{and} \qquad g \colon B \mapsto B^{c \prime}  
 \]
 form an order-preserving Galois-connection $(f,g)$ on $(\wp(U),\subseteq)$. If we set $R = I^c$,
 then $R$ is obviously a tolerance and $\mathbb{K} = (U,U,R^c)$. By our above observations 
 $f(A) = A^\UP$ and $g(A) = A^\DOWN$ for all $A \subseteq U$, and $L \cong  \mathfrak{B}(\mathbb{K})  \cong \wp(U)^\DOWN \cong \wp(U)^\UP$.
\end{proof}

The lattice $\wp(U)^\DOWN$ is not necessarily even modular; for instance, in Example~\ref{Ex:Counter}
(p.~\pageref{Ex:Counter}), we define a tolerance $R$ on the set $U = \{a,b,c,d,e\}$ such that
$\wp(U)^\DOWN = \{ \emptyset, \{a\}, \{c\}, \{e\}, \{a,b\}, \{a,e\}, \{d,e\}, \{a,b,c\}, \{c,d,e\}, U \}$.
Now, the set 
$\{ \emptyset, \{a\}, \{c\}, \{a,b\}, \{a,b,c\} \}$ forms a sublattice of $\wp(U)^\DOWN$
isomorphic to $\mathbf{N_5}$. 

A complete lattice $L$ is \emph{completely distributive} if for any doubly indexed
subset $\{x_{i,\,j}\}_{i \in I, \, j \in J}$ of $L$, we have
\[
\bigwedge_{i \in I} \Big ( \bigvee_{j \in J} x_{i,\,j} \Big ) = 
\bigvee_{ f \colon I \to J} \Big ( \bigwedge_{i \in I} x_{i, \, f(i) } \Big ), \]
that is, any meet of joins may be converted into the join of all
possible elements obtained by taking the meet over $i \in I$ of
elements $x_{i,\,k}$\/, where $k$ depends on $i$.

In \cite{ganter1999formal}, Theorem 40 presents the following condition equivalent to
the assertion that the concept lattice $\mathfrak{B}(\mathbb{K})$ is completely distributive:

\begin{itemize}
\item[($\dag$)] For every non-incident object-attribute pair $(g,m)\notin I$, there
exists an object $h\in G$ and an attribute $n\in M$ with $(g,n)\notin
I,(h,m)\notin I$, and $h\in k^{\prime\prime}$ for all $k\in G\setminus \{n\}^{\prime}$.
\end{itemize}

We know by Proposition~\ref{Prop:ConceptIsom} that for any tolerance $R$ on $U$, 
$\wp(U)^\UP$ and $\wp(U)^\DOWN$ are isomorphic
to the concept lattice of the context  $\mathbb{K} = (U,U,R^c)$. If $\mathbb{K}$
is identified with $(G,M,I)$, then for all $x,y \in U$, $(x,y) \notin I$ means that  
$x \, R \, y$ and $y \, R \, x$. 
Since $x'' = \{x\}^{\UP\DOWN} = R(x)^\DOWN$, $y \in x''$ means that $R(y) \subseteq R(x)$.
In addition, $U \setminus \{x\}' = \{x\}^{\prime \, c} = \{x\}^\UP = R(x)$.
Hence ($\dag$) is equivalent to the following condition:

\begin{itemize}
\item[($\ddag$)] For any $a \, R \, b$, there exist $c,d\in U$ with
$a \, R \, c$ and $b \, R \, d$ such that for all $k\in R(c)$, we have $R(d)\subseteq R(k)$.
\end{itemize}

In what follows, we are going to present some conditions equivalent to  ($\ddag$).
Let $R$ be a tolerance on $U$. A set $X\subseteq U$ is a \emph{preblock} of
$R$ if $a \, R \, b$ holds for all $a,b\in X$, that is, $X^{2} \subseteq R$, 
where $X^{2}$ means the Cartesian product $X \times X$. A \emph{block} of $R$ is a maximal
preblock $B$. It is well known that $B = \bigcap_{x\in B} R(x)$, and that any preblock 
is contained in some block of $R$  (see e.g. \cite{Shreider}). Hence, for any
$x,y \in U$, $x \, R \, y$ if and only if there exists a block $B$
such that $x,y \in B$.

Denoting the set of blocks of $R$ by
$\mathcal{B}(R)$, we obtain $R = \bigcup \{ B^{2} \mid B \in \mathcal{B}(R) \}$.
A collection $\mathcal{H} \subseteq \wp(U)$ of nonempty subsets of $U$ is called a 
\emph{covering} of $U$ if $\bigcup \mathcal{H} = U$. A covering $\mathcal{H}$ 
is \emph{irredundant} if $\mathcal{H} \setminus \{X\}$ is not a covering of $U$ for any 
$X\in \mathcal{H}$. Clearly, the blocks of any tolerance on form a covering, which is 
not in general irredundant. Conversely, for
any covering $\mathcal{H}$ of $U$, the relation 
$R_{\mathcal{H}}= \bigcup \{ X^{2} \mid X \in \mathcal{H\}}$ 
is tolerance on $U$, called the \emph{tolerance induced by $\mathcal{H}$}.
Note that $\mathcal{H} \subseteq \mathcal{B}(R_\mathcal{H})$ for any irredundant covering $\mathcal{H}$ and that 
this inclusion can be proper (see Example~\ref{Ex:Schreider}).

In \cite{Pomykala88}, J.~A.~Pomyka{\l}a presented the following definition of covering-based rough approximations.
Let $\mathcal{H} \subseteq U$ be a covering and denote $\mathcal{H}_x = \bigcup \{ B \in \mathcal{H} \mid x \in B\}$
for any $x \in U$. The approximations of any $X \subseteq U$ are defined by
\[ \underline{\mathcal{H}}(X) = \{ x \in U \mid \mathcal{H}_x \subseteq X \} \text{ \ and \ }
\overline{\mathcal{H}}(X) = \bigcup \{ B \in \mathcal{H} \mid B \cap X \ne \emptyset \}.
\]
If $R$ is a tolerance on $U$ and $\mathcal{H}$ is the family of blocks of $R$, that is, 
$\mathcal{H} = \mathcal{B}(R)$, then for all $X \subseteq U$, $\underline{\mathcal{H}}(X) = X^\DOWN$
and $\overline{\mathcal{H}}(X) = X^\UP$. Note also that $R(x) = \mathcal{H}_x$ for any $x \in U$ (see \cite{Jarv99} for details).

\begin{theorem} \label{Thm:DC}
Let $R$ be a tolerance on $U$. The following assertions are equivalent:
\begin{enumerate}[\rm (a)]
\item $R$ satisfies {\rm ($\ddag$)}.

\item For any $a \, R \, b$, there exists $d\in U$ with $R(d)\subseteq R(a)\cap R(b)$ such that for
all $x \, R \, d$, we have $R(d) \subseteq R(x)$.

\item For any $a \, R \, b$, there exists a block $B \in \mathcal{B}(R)$ and an element 
$d\in B$ such that $a,b \in R(d) = B$.

\item $R$  is a tolerance induced by an irredundant covering of $U$. 
\end{enumerate}
\end{theorem}

\begin{proof}
(a)$\Rightarrow$(b): Let $R$ be a tolerance satisfying ($\ddag$). Then, 
$a \, R \, b$ implies that there are $c\in R(a)$ and $d\in R(b)$ such that 
$R(d)\subseteq R(c)$ and $R(d)\subseteq R(a)$.
Additionally, $b \in R(d)$ implies $b \in R(c)$, which gives $R(d) \subseteq R(b)$. 
Thus, $R(d) \subseteq R(a)\cap R(b)$. 
Finally, if $x \, R \, d$, then $x \in R(d) \subseteq R(c)$, and hence $R(d) \subseteq R(x)$
by  ($\ddag$).

(b)$\Rightarrow$(c): Suppose $R$ satisfies (b). If $a \, R \, b$, then there
is an element $d\in U$ with $R(d)\subseteq R(a)\cap R(b)$. Thus, also $a \, R \, d$ and $b \, R \, d$
hold, and we have $\{a,b,d\}^{2}\subseteq R$. Hence, there is a block $B\in$
$\mathcal{B}(R)$ with $a,b,d\in B$. Note that since $R$ satisfies (b), 
we have  $x \, R \, d$ and $R(d) \subseteq R(x)$ for all $x\in B$. Thus, we get
\[
B\subseteq R(d) \subseteq \bigcap_{x\in B} R(x) = B,
\]
and hence $a,b \in R(d) = B$.

(c)$\Rightarrow$(d): Suppose that (c) holds and let us define a
family $\mathcal{K}$ of blocks by 
$\mathcal{K} = \{B \in \mathcal{B}(R) \mid B= R(d) \text{ for some $d\in U$} \}$. 
Now, for all $x\in U$, $x \, R \, x$ implies that there is a block 
$B \in\mathcal{K}$ containing $x$, and so $\mathcal{K}$ is a
covering of $U$. In view of (c), for any $a \, R \, b$, there is a block
$B\in\mathcal{K}$ with $a,b\in B$. Hence, 
$R\subseteq \bigcup_{B\in\mathcal{K}} B^{2}\subseteq R$, giving 
$R = \bigcup_{B\in\mathcal{K}} B^{2}$. Thus, $R$ is induced by the covering 
$\mathcal{K}$. Finally, we show
that $\mathcal{K}$ is irredundant by proving that 
$\bigcup (\mathcal{K} \setminus\{B\}) \neq U$ for any $B \in \mathcal{K}$. 
Indeed, if $B \in \mathcal{K}$, then there exists $d \in U$ such that $R(d) = B$. 
Suppose that $d \in X$ for some block $X \in \mathcal{K} \setminus\{B\}$. 
Then, $d \, R \, x$ for all $x\in X$, whence we get $X \subseteq R(d) = B$. 
Since $X$ and $B$ are blocks, we obtain $X = B$, a contradiction. 
Thus, $d \notin \bigcup (\mathcal{K}\setminus\{B\})$.

(d)$\Rightarrow$(b): Assume that $R = \bigcup \{ X^{2} \mid X \in \mathcal{H} \}$, 
where $\mathcal{H}$ is an irredundant covering of $U$ and suppose that $a \,R \, b$. 
Then, there exists $X \in \mathcal{H}$ such that $a,b \in X$, and clearly, $x \,R \, y$ for all $x,y\in X$.
Hence, $X \subseteq R(x)$ for all $x \in X$. Since 
$\bigcup (  \mathcal{H} \setminus \{X\} ) \neq U$, there is $d \in X$ such
that $d \notin Y$ for all $Y \in \mathcal{H} \setminus \{X\}$. 
Observe that $d \, R \, y$ for some $y \in U \setminus X$ would imply that 
$\{d,y\}$ is contained in some block $Y \in \mathcal{H}$ different from $X$. 
Since this is impossible, we get $R(d) \subseteq X \subseteq R(x)$ 
for all $x \in X$. In particular, we obtain $R(d)\subseteq R(a)\cap R(b)$, 
and also $R(d)\subseteq R(x)$ for each $x\in U$ with $x \, R \, d$, 
because $x \, R \, d$ implies $x \in R(d) \subseteq X$.

(b)$\Rightarrow$(a): If (b) holds, then ($\ddag$) is satisfied with $c = d$.
\end{proof}

Our next proposition is now clear by the above-mentioned observations.

\begin{proposition} \label{Prop:ComplDistributive}
Let $R$ be a tolerance on a set $U$. The complete lattices $\wp(U)^\DOWN$ and
$\wp(U)^\UP$ are completely distributive if and only if
$R$ is induced by an irredundant covering of $U$.
\end{proposition}

\begin{remark} \label{Rem:Block}
If $a \, R \, b$ holds for any $a,b \in R(x)$, then $R(x)$ is a block. Namely, 
in this case $R(x)= \bigcap_{a \in R(x)} R(a)$, which means that 
$R(x)$ is block of $R$.
\end{remark}

Next we characterize irredundant coverings in terms of tolerances they induce.

\begin{proposition} \label{Prop:IrredundantCovering}
Let $R$ be a tolerance induced by a covering $\mathcal{H} \subseteq \wp(U)$. Then, the 
following assertions are equivalent:
\begin{enumerate}[\rm (a)]
 \item $\mathcal{H}$ is an irredundant covering;
 \item For each $B \in \mathcal{H}$, there exists $d \in U$ such that $R(d)= B$.
\end{enumerate}
\end{proposition}

\begin{proof}
(a)$\Rightarrow$(b): Let $\mathcal{H}$ be an irredundant covering and $B \in \mathcal{H}$. This means
that there is $d \in B$ such that $d \notin \bigcup ( \mathcal{H} \setminus \{B\} )$. Since $R$ is
induced by $\mathcal{H}$, we have $d \, R \, b$ for all $b \in B$. Thus, $B \subseteq R(d)$.
On the other hand, if $d \, R \, x$, then there is a block $B' \in \mathcal{H}$ such that $d,x \in B'$.
Since $d \notin \bigcup ( \mathcal{H} \setminus \{B\} )$, we obtain $B' = B$ and $x \in B$.
Thus, also $R(d) \subseteq B$ and $R(d) = B$.

(b)$\Rightarrow$(a): It is enough to show that $d \notin \bigcup ( \mathcal{H} \setminus \{B\} )$,
where $B$ belongs to $\mathcal{H}$ and $d$ is such that $R(d) = B$. Assume for contradiction
that there exists $B ' \in \mathcal{H}$ such that $B' \neq B$ and $d \in B'$. By (b),
there is $d'$ such that $R(d') = B'$. Since $d,d' \in B'$, we have $d \, R \, d'$ and $d, d' \in B$. 
Hence, if $a \in B$, then $a \, R \, d'$ and $a \in R(d') = B'$, that is,
$B \subseteq B'$. Similarly, $a \in B'$ implies $a \, R \, d$ and $a \in R(d) = B$, giving
$B' \subseteq B$. Therefore, $B = B'$, a contradiction.
\end{proof}

\begin{example} \label{Ex:InfSyst1}
In this example, we consider how tolerances induced by an irredundant covering
arise in information systems. By Proposition~\ref{Prop:IrredundantCovering}, the essential condition
is that for each member $B$ of a covering $\mathcal{H}$, there exists $d \in U$ such that $R(d)= B$.

\medskip\noindent%
(a) Information systems introduced by Pawlak \cite{pawlak1981information} are triples
$\mathcal{S} = (U,A,\{V_a\}_{a \in A})$, where $U$ is a nonempty set of \emph{objects}, 
$A$ is nonempty set of \emph{attributes}, and $\{V_a\}_{a \in A}$ is an
indexed set of \emph{value sets of attributes}. Each attribute is a function $a \colon U \to V_a$.
In real-world situations, some attribute values for an object may be missing. In 
\cite{Kryszkiewicz1998} these \emph{null values} are dealt with marking them
by $*$. This kind of information systems are called \emph{incomplete information systems}.
For each $B \subseteq A$, the following tolerance is defined:
\[
 \mathit{sim}_B = \{ (x,y) \in U \times U \mid (\forall a \in A) \, a(x) = a(y) \text{ or }
 a(x) = * \text{ or } a(y) = * \}.
\]
Additionally, let us denote by
\[ 
\mathit{compl}_B(U) = \{ x \in U \mid a(x) \neq * \text{ for all } a \in B \} 
\]
the set of \emph{$B$-complete elements}. 

Let us assume that each element of $U$ is $\mathit{sim}_B$-related to at least one $B$-complete element.  
Then, the family 
$\mathcal{H}_B = \{ \mathit{sim}_B(x) \mid x \in \mathit{compl}_B(U) \}$ is a covering and $\mathcal{H}$
is clearly irreducible, because each $B$-complete element $x$ can belong to only to  
$\mathit{sim}_B(x)$. It is also obvious that $\mathit{sim}_B$ is induced by 
$\mathcal{H}_B$.

\medskip\noindent%
(b) Another way to present incomplete information is to use nondeterministic information systems (see e.g. \cite{Jarv99,DemOrl02}). 
A \emph{nondeterministic information system} $\mathcal{S} = (U,A,\{V_a\}_{a \in A})$ is such that each attribute is a function 
$a \colon U \to \wp(V_a) \setminus \emptyset$. 
Note that Pawlak's ``original'' information systems considered in (a) can be viewed as nondeterministic 
information systems such that $|a(x)| = 1$ for all $x \in U$ and $a \in A$.

We may interpret a nondeterministic information system $\mathcal{S}$ as a so-called \emph{approximate information system} in
such a way that for an object $x \in U$ and an attribute $a \in A$, 
the unique value of the attribute $a$ for the object $x$ is assumed to be in $a(x)$; complete ignorance is
denoted by $a(x) = V_a$.

Let $B \subseteq A$. Resembling case (a), we say that an object $x$ is \emph{$B$-complete} 
if $a(x)$ is a singleton for all $a \in B$. In other words, for a $B$-complete object, all its $B$-values are known precisely, 
without ambiguity. These elements can be considered as leaning examples. 
Let us again denote by $\mathit{compl}_B(U)$ the set of $B$-complete elements. Note that for all $B$-complete elements $x$ and $a \in A$, 
we can simply write $a(x) = v$, that is, $a$ behaves like a single-valued attribute for them.

For any $B \subseteq A$, we can now define the following relation $R_B$:
\[ (x,y) \in R_B \iff \text{exists } c \in \mathit{compl}_B(U) \text{ such that } a(c) \in a(x) \cap a(y) \text{ for all $a \in B$} .\]
This means that two objects $x$ and $y$ are $R_B$-related if and only if there is a completely known element $c$, which
is ``potentially'' $B$-indiscernible with $x$ and $y$. 

It is easy to see that for all $B \subseteq A$, the relation $R_B$ satisfies conditions (b) and (c) of Theorem~\ref{Thm:DC}.
The learning examples $c \in \mathit{compl}_B(U)$ are such that $\{ R_B(c) \mid c \in \mathit{compl}_B(U)\}$ forms an irredundant covering and the relation
induced by this covering is $R_B$. 

Note that both in (a) and (b), if all elements of $U$ are $B$-complete, then the relations $\mathit{sim}_B$ and
$R_B$ are equivalences, and the corresponding irredundant covering is a partition corresponding these equivalences.
\end{example}

Let $L$ be a lattice with a least element $0$. 
The lattice $L$ is \emph{atomistic}, if any element of $L$ is the join of
atoms below it. It is well known (see e.g.\@ \cite{Grat98}) that a 
complete Boolean lattice is atomistic if and only if it is completely distributive.

\begin{proposition} \label{Prop:Boolean}
Let $R$ be a tolerance induced by an irredundant covering of $U$. The complete lattices
$\wp(U)^\DOWN$ and  $\wp(U)^\UP$ are atomistic Boolean lattices such that
$\{ R(x)^\DOWN \mid R(x) \text{ is a block}\, \}$ and $\{ R(x) \mid R(x) \text{ is a block}\, \}$
are their sets of atoms, respectively.
\end{proposition}

\begin{proof}
By Proposition~\ref{Prop:ComplDistributive}, $\wp(U)^\DOWN$ and
$\wp(U)^\UP$ are completely distributive. 
Because they are ortholattices also, they are Boolean lattices, and
since they are completely distributive, they must be atomistic.

Atoms of  $\wp(U)^\UP$ need to be of the form $R(x)$, because the map $^\UP$ is
order-preserving and $R(x) = \{x\}^\UP$. Since $^\DOWN$ is an isomorphism from
$\wp(U)^\UP$ to $\wp(U)^\DOWN$, the atoms of  $\wp(U)^\DOWN$ are of the form 
$R(x)^\DOWN$.

Suppose that $R(x)$ is a block and $R(y) \subseteq R(x)$. Because $y \in R(x)$
and $R(x)$ is a block, we must have $R(y) = R(x)$, and
hence $R(x)$ is an atom of  $\wp(U)^\UP$. Analogously,  $R(x)^\DOWN$
is an atom of $\wp(U)^\DOWN$, whenever $R(x)$ is a block.

On the other hand, suppose $R(x)$ is an atom of  $\wp(U)^\UP$; then 
$R(x)^\DOWN$ is an atom of  $\wp(U)^\DOWN$. Suppose that  $a,b \in R(x)$.
Since $a \, R \, x$, by Theorem~\ref{Thm:DC}, there exists $d \in U$ with
$R(d)\subseteq R(a) \cap R(x)$, and so  $d \in R(a)^\DOWN \cap R(x)^\DOWN$.
Hence, $\emptyset \subset R(a)^\DOWN \cap R(x)^\DOWN \subseteq R(x)^\DOWN$, and
because  $R(x)^\DOWN$ is an atom, 
we have $R(a)^\DOWN \cap R(x)^\DOWN = R(x)^\DOWN$, that is,
$R(x)^\DOWN \subseteq R(a)^\DOWN$. Now $x \in R(x)^\DOWN \subseteq R(a)^\DOWN$
implies $b \in R(x) \subseteq R(a)$. Thus, $a \, R \, b$, and so, 
by Remark~\ref{Rem:Block}, $R(x)$ is a block. 
\end{proof}

\begin{example}\label{Ex:Schreider}
Let $A = \{1,2,\ldots,n\}$ be a finite set. We define a tolerance $R$ on 
$U = \wp(A) \setminus \{\emptyset\}$ by setting for any nonempty subsets $B,C\subseteq A$:
\[
(B, C) \in R \iff B \cap C \neq \emptyset.
\]
The structure $(U,R)$ is called an \emph{$n-1$-dimensional simplex} (see \cite{Shreider}). 
Let $i\in A$ and define the set $K_{i} = \{B \in U  \mid i\in B\}$.
Clearly, $R(\{i\})=K_{i}$, and it is easy to see that $K_{i}$ is also a
tolerance block. Now let $\mathcal{H} = \{K_{1},K_{2},...,K_{n}\}$. Then,
$(B, C) \in R$ means that $B$ and $C$ have a common element $j$ and 
$(B,C)\in {K_j}^2$. Hence, 
\[
 R = {K_1}^2 \cup {K_2}^2 \cup \cdots \cup {K_n}^2.
\]
Clearly, $\wp(A)\setminus \{\emptyset\} = K_{1} \cup K_{2} \cup...\cup K_{n}$, that is,
$\mathcal{H}$ is a covering of $U$. This covering is irredundant, because
if we omit $K_j$, then the set $\{j\}$ cannot be covered.
For instance, if $n = 3$, then $K_1 = \{ \{1\}, \{1,2\}, \{1,3\}, U$\},
$K_2 = \{ \{2\}, \{1,2\}, \{2,3\}, U$\}, and  $K_3 = \{ \{3\}, \{1,3\}, \{2,3\}, U$\}.
Also the set $\{ \{1,2\}, \{1,3\}, \{2,3\}, U \}$ is a block of $R$, showing
that $\mathcal{H} \subset \mathcal{B}(R_\mathcal{H})$.
Because $\wp(U)$ is finite, by Proposition~\ref{Prop:Boolean} this means that  $\wp(U)^\DOWN$ and $\wp(U)^\UP$  
are finite Boolean lattices.
\end{example}

For a tolerance $R$, an \emph{$R$-path} is a sequence $a_0, a_1, \ldots,a_n$ of distinct elements of $U$ 
such that $a_i \, R \, a_{i+1}$ for all $0 \leq i \leq n-1$. The \emph{length} of a path 
is the number of elements in the sequence minus one. Note that each point of $U$ forms a path of length zero.
We denote by $\overline{R}$ the transitive closure of $R$, that is,
\[ \overline{R} = R \cup R^2 \cup R^3 \cup \cdots \cup R^n \cup \cdots \ . \]
Then, $\overline{R}$ is the smallest equivalence containing $R$.
Note that for any $X \subseteq U$, the upper $\overline{R}$-approximation ${\overline{R}}(X)$
of $X$ consists of the elements that are connected to at 
least one element in $X$ by an $R$-path. 
We denote ${\overline{R}}(X)$ simply by $\overline{X}$. 

For any $a \in \overline X$, we define the \emph{distance of $a$ from $X$}, denoted 
$\delta(a,X)$, as the minimal length of an $R$-path connecting $a$ at least to one element 
in $X$. Note that if $a \in X$, then $\delta(a,X) = 0$. For any $n \geq 0$, let us define the set
\[ X^n = \{ a \in U \mid \delta(a,X) = n \} .\]
The above means that $\overline{X} = \bigcup_{n \geq 0} X^n$. In addition, we denote:
\begin{align*}
X_{\rm  even} &= X^0 \cup X^2 \cup X^4 \cup \cdots \cup X^{2k} \cup \cdots \  \\
X_{\rm odd}   &= X^1 \cup X^3 \cup X^5 \cup \cdots \cup X^{2k+1} \cup \cdots \ 
\end{align*}
Our the next lemma presents some properties of $X_{\rm  even}$ and $X_{\rm odd}$
that will be needed in the proofs of the next section.

\begin{lemma}\label{Lem:OddEven}
If $\rho$ be a tolerance on $U$, then the
following assertions hold for all $X \subseteq U$:
\begin{enumerate}[\rm (a)]
\item $\overline{X}$ is a disjoint union of $X_{\rm even}$ and $X_{\rm odd}$. 
\item $X\subseteq X_{\rm even}$ and $\rho(X) \setminus X \subseteq X_{\rm odd}$.
\item $X_{\rm odd} \subseteq \rho(X_{\rm even}) = \overline{X}$ and $\overline{X} \setminus X \subseteq \rho(  X_{\rm odd})$.
\item $X \subseteq \rho(\overline{X}\setminus X)$ implies
$\rho(X_{\rm odd}) = \rho(X_{\rm even}) = \overline{X}$.
\end{enumerate}
\end{lemma}

\begin{proof} Claim (a) is obvious.

(b) Clearly $X = X^0 \subseteq X_{\rm even}$, and $\rho(X) = X \cup X^1$ implies
$\rho(X) \setminus X = X^1 \subseteq X_{\rm odd}$.

(c) Suppose that $a \in X_{\rm odd}$. Then, there is a $\rho$-path $(a_0,\ldots,a_{2k},a_{2k+1})$ of length
$2k + 1$ such that $a_0 \in X$ and $a_{2k + 1} = a$. So,  $(a_0,\ldots,a_{2k})$ is a $\rho$-path of length
$2k$ and thus $a_{2k} \in X_{\rm even}$. Now $(a, a_{2k}) \in \rho$ gives $a \in \rho(X_{\rm even})$. 
In addition,  $\overline{X} = X_{\rm odd} \cup X_{\rm even} \subseteq \rho(X_{\rm even}) \subseteq \overline{X}$.
For the other part, let $a \in \overline{X} \setminus X$. Then, we have $a \in X^n$ for some $n \geq 1$.
If $n$ is an odd number, then $a \in X_{\rm odd} \subseteq  \rho(X_{\rm odd})$. If $n$ is even,
then $n-1$ is odd and $a \in \rho(X^{n-1}) \subseteq \rho(X_{\rm odd})$.

(d) Assume $X \subseteq \rho(\overline{X} \setminus X)$. Then, for each $x \in X$, there is 
$y \in \overline{X} \setminus X$ such that $x \, \rho \, y$. Since $y \in X^1 \subseteq X_{\rm odd}$,
we get $x \in \rho(X_{\rm odd})$ and  $X \subseteq  \rho( X_{\rm odd})$. Thus,
$\overline{X} = X \cup (\overline{X} \setminus X) \subseteq \rho(X_{\rm odd}) \subseteq \overline{X}$.
\end{proof}

\section{Lattice structures of rough sets determined by tolerances}
\label{Sec:OrderedSets}

Let $R$ be a tolerance on $U$.
We begin by considering the set  $\wp(U)^\DOWN \times \wp(U)^\UP$ ordered coordinatewise
by $\subseteq$. It is clear that  $\wp(U)^\DOWN \times \wp(U)^\UP$
is a complete lattice such that
\begin{align}
\bigwedge_{i \in I} (A_i,B_i) = 
\Big ( \bigcap_{i \in I} A_i, \big (\bigcap_{i \in I} B_i \big )^{\DOWN \UP} \Big )
\intertext{and}
\bigvee_{i \in I} (A_i,B_i) = 
\Big (  \big (\bigcup_{i \in I} A_i, \big )^{\UP \DOWN}, \bigcup_{i \in I} B_i  \Big )
\end{align}
for all $(A_i,B_i)_{i \in I} \subseteq \wp(U)^\DOWN \times \wp(U)^\UP$. Let us
define a map ${\sim}$ on $\wp(U)^\DOWN \times \wp(U)^\UP$ by setting for any $(A,B) \in 
\wp(U)^\DOWN \times \wp(U)^\UP$:
\begin{equation}\label{Eq:Negation}
 {\sim}(A,B) = (B^c, A^c).
\end{equation}
The map ${\sim}$ can be viewed as a so-called \emph{De~Morgan operation}, 
because it satisfies for all $(A,B) \in \wp(U)^\DOWN \times \wp(U)^\UP$,
\[
{\sim}{\sim}(A,B) = (A,B)
\]
and
\[(A_1, B_1) \leq (A_2,B_2) \text{ if and only if } 
{\sim} (A_1, B_1)  \geq {\sim} (A_2, B_2) .
\]
In this work, lattices with a De~Morgan operation are called \emph{polarity lattices}.
Note that so-called De~Morgan lattices/algebras are usually distributive, but polarity
lattices are generally not \cite{Birk95}.

As in the case of equivalences, the set of all rough sets is denoted by 
$\mathit{RS} = \{ (X^\DOWN, X^\UP) \mid X \subseteq U \}$. Obviously, 
$\mathit{RS} \subseteq \wp(U)^\DOWN \times \wp(U)^\UP$ and $\mathit{RS}$ is
bounded with $(\emptyset,\emptyset)$ as the least element and $(U,U)$
as the greatest element. It is known \cite{Jarv07} that $\mathit{RS}$ is self-dual by the map 
${\sim}$, and for all $X \subseteq U$, 
\[ {\sim}(X^\DOWN, X^\UP) = (X^{\UP c}, X^{\DOWN c}) = (X^{c \UP}, X^{c \DOWN}).\] 

Recall that the map $X \mapsto X^{\UP \DOWN}$ is a closure operator and $X \mapsto X^{\DOWN \UP}$ 
is an interior operator (see Section~\ref{Sec:ToleranceApproximations}). 
Our next lemma shows how $\mathit{RS}$ can be also represented up to isomorphism 
as interior-closure pairs.

\begin{lemma} \label{Lem:Isomorphic}
If $R$ is a tolerance, then $\mathit{RS} \cong \{ (X^{\DOWN\UP}, X^{\UP\DOWN}) \mid X \subseteq U \}$.
\end{lemma}

\begin{proof}
We show that the map $\varphi \colon (X^\DOWN,X^\UP) \mapsto  (X^{\DOWN\UP}, X^{\UP\DOWN})$ is an order-isomorphism.
If $(X^\DOWN,X^\UP) \leq (Y^\DOWN,Y^\UP)$, then $X^\DOWN \subseteq Y^\DOWN$ implies 
$X^{\DOWN\UP} \subseteq Y^{\DOWN\UP}$. Similarly, $X^\UP \subseteq Y^\UP$ gives 
$X^{\UP\DOWN} \subseteq Y^{\UP\DOWN}$.
Thus, $(X^{\DOWN\UP}, X^{\UP\DOWN}) \leq (Y^{\DOWN\UP}, Y^{\UP\DOWN})$.

On the other hand, if  $(X^{\DOWN\UP}, X^{\UP\DOWN}) \leq (Y^{\DOWN\UP}, Y^{\UP\DOWN})$,
then $X^{\DOWN\UP} \subseteq Y^{\DOWN\UP}$ implies $X^\DOWN = X^{\DOWN\UP\DOWN} \subseteq 
Y^{\DOWN\UP\DOWN} = Y^\DOWN$, and from $X^{\UP\DOWN} \subseteq Y^{\UP\DOWN}$ we get
$X^\UP = X^{\UP\DOWN\UP} \subseteq Y^{\UP\DOWN\UP} = Y^\UP$. 
So, $(X^\DOWN,X^\UP) \leq (Y^\DOWN,Y^\UP)$.

Thus, $\varphi$ is an order-embedding. The map $\varphi$ is surjective,
because any pair $(X^{\DOWN\UP}, X^{\UP\DOWN})$ is the image of
$(X^\DOWN,X^\UP) \in \mathit{RS}$.
\end{proof}

However, not any pair of interior and closure operators defines up to isomorphism 
the same structure as rough sets determined by tolerances.  We present the following characterization 
of rough sets in terms of interior and closure operators.

\begin{proposition}\label{Prop:ClosureRepr}
Let $\mathcal{I}$ and $\mathcal{C}$ be lattice-theoretical 
interior and closure operators on the set $U$. Then, there exists a tolerance on U such that 
$\mathit{RS} \cong \{ ( \mathcal{I}(X), \mathcal{C}(X) ) \mid X \subseteq U\}$ if
and only if there exists a Galois connection $(F,G)$ on $\wp(U)$
such that $\mathcal{C} = G \circ F$, \ $\mathcal{I} = F \circ G$ 
and the following conditions hold for all $x,y \in U$:
\begin{enumerate}[\rm (i)]
 \item $x \in F(\{x\})$;
 \item $x \in F(\{y\})$ implies $y \in F(\{x\})$.
\end{enumerate}
\end{proposition}

\begin{proof}
($\Rightarrow$)\, Let $R$ be a tolerance on $U$. 
We denote the closure operator $X \mapsto X^{\UP\DOWN}$ by $\mathcal{C}$ 
and the interior operator $X \mapsto X^{\DOWN\UP}$ by $\mathcal{I}$. Then,
$\mathit{RS}$ is order-isomorphic to $\{ ( \mathcal{I}(X), \mathcal{C}(X) ) \mid X \subseteq U\}$
by Lemma~\ref{Lem:Isomorphic}. Because $R(x) = \{x\}^\UP$  for all $x \in U$, 
conditions (i) and (ii) hold.

($\Leftarrow$)\, Suppose that $(F,G)$ is a Galois connection satisfying $\mathcal{C} = G \circ F$, 
$\mathcal{I} = F \circ G$, and that conditions (i) and (ii) hold for $F$. 
Let us define a relation $R$ by setting $R(x) = F(\{x\})$.
As in the proof of Proposition~\ref{Prop:GaloisCharacterization}, we can see that
$R$ is a tolerance on $U$ such that $X^\UP = F(X)$ and $X^\DOWN = G(X)$ for all $X \subseteq U$.
It is now clear that for all $X \subseteq U$, $\mathcal{I}(X) = X^{\DOWN\UP}$ and
$\mathcal{C}(X) = X^{\UP\DOWN}$, and
$\mathit{RS} \cong \{ ( \mathcal{I}(X), \mathcal{C}(X) ) \mid X \subseteq U\}$ 
follows from Lemma~\ref{Lem:Isomorphic}.
\end{proof}

If $|U| \leq 4$, then $\mathit{RS}$ is a lattice, but when $|U| \geq 5$, 
$\mathit{RS}$ does not necessarily form a lattice, as can be seen
in the following example; see \cite{Jarv99,Jarv01}.

\begin{example} \label{Ex:Counter}
Let $R$ be a tolerance on $U = \{a,b,c,d,e\}$ such that
$R(a) = \{a,b\}$, $R(b) = \{a,b,c\}$, $R(c) = \{b,c,d\}$, $R(d) = \{c,d,e\}$, and
$R(e) = \{d,e\}$. The ordered set $\mathit{RS}$ is depicted in Figure~\ref{Fig:Fig1}.

\begin{figure}[ht]
\centering
\includegraphics[width=\textwidth]{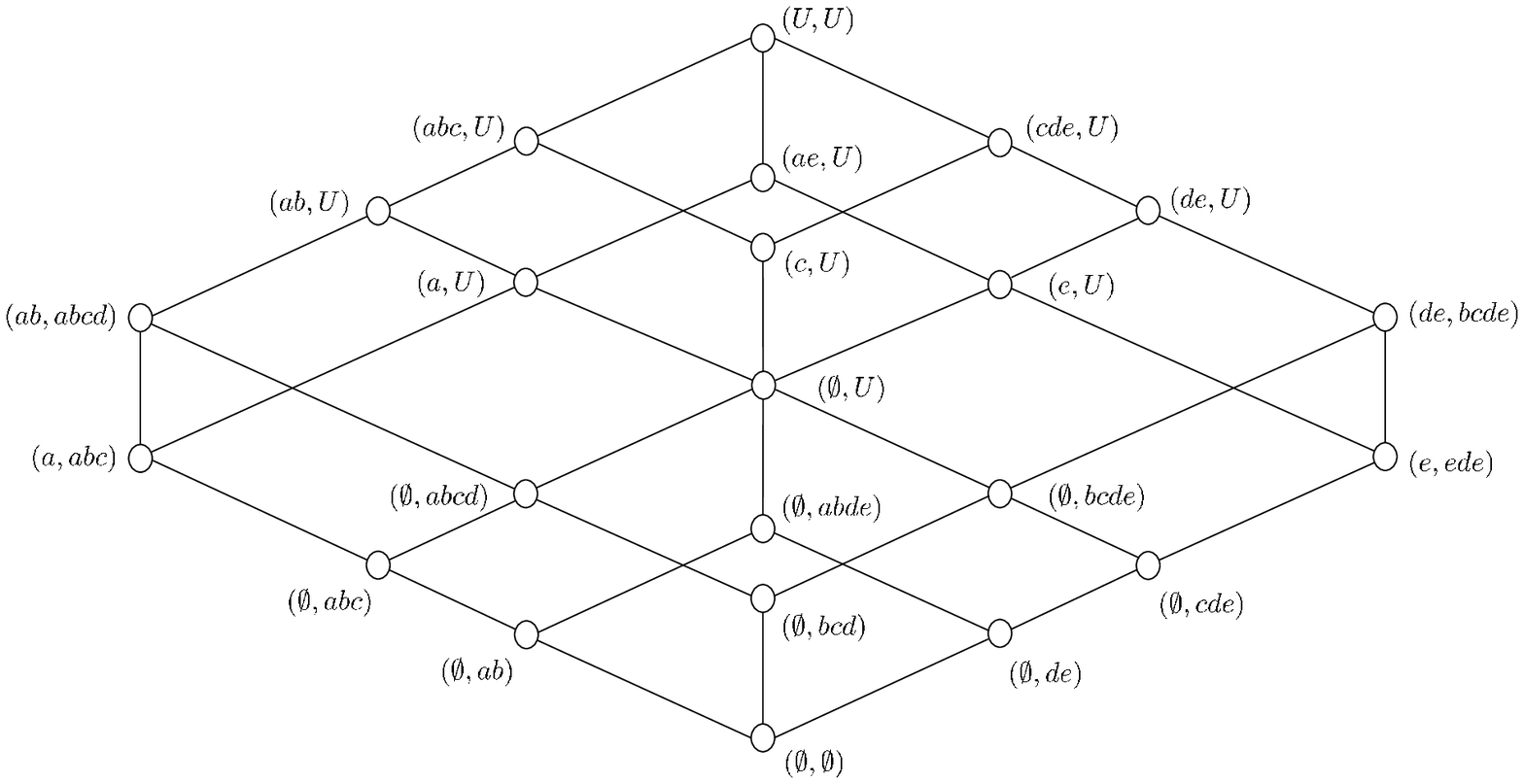} 
\caption{\label{Fig:Fig1}}
\end{figure}
For instance, the elements $(a, abc)$ and $(\emptyset, abcd)$ do not have a least upper bound. 
Similarly, $(ab, abcd)$ and $(a, U)$ do not have the greatest lower bound.
\end{example}

Next we consider completions of $\mathit{RS}$. Let us define the set 
\begin{equation}\label{EQ:DefineS}
 \mathcal{S} = \{x \in U \mid R(x)=\{x\}\},
\end{equation}
that is, $\mathcal{S}$ consists of such $x$'s that $R(x)$ is a singleton.
Then, $\{x\}^{\DOWN} =\{x\}$ for all $x\in \mathcal{S}$ and
$\mathcal{S}^\DOWN = \mathcal{S} = \mathcal{S}^\UP$. 
Clearly, for any $X \subseteq U$ and $x \in \mathcal{S}$,
\[
x\in X^\UP \iff R(x) \cap X \neq \emptyset \iff R(x)\subseteq X \iff x\in X^\DOWN.
\]
Hence, for all $x \in \mathcal{S}$, we have either $x\in X^{\DOWN}$ or $x \in X^{\UP c}$. 
This means that $\mathcal{S} \subseteq X^{\DOWN} \cup X^{\UP c}$ holds for any $X\subseteq U$.
We define the set of pairs
\[
\mathcal{I}(\mathit{RS}) = \{ (A,B) \in \wp(U)^{\DOWN} \times \wp(U)^{\UP}
\mid A^{\UP}\subseteq B^\DOWN \text{ and } \mathcal{S} \subseteq A\cup B^{c}\}.
\]
Because  $(X^{\DOWN})^{\UP} \subseteq X  \subseteq (X^{\UP})^{\DOWN}$ and 
$\mathcal{S} \subseteq X^\DOWN \cup X^{\UP c}$
for any $(X^\DOWN,X^\UP) \in \mathit{RS}$, we have $RS\subseteq\mathcal{I}(\mathit{RS})$.

The \emph{Dedekind--MacNeille completion} of an ordered set $P$ can 
be defined as the smallest complete lattice with $P$ order-embedded in it
(see \cite{DaPr02}, for example). In \cite{UmaThesis}, D.~Umadevi  
proved that for a reflexive relation $R$ on $U$, the Dedekind--MacNeille completion
of $\mathit{RS}$ is 
\[
\mathcal{DM}(\textit{RS}) = \{ (A,B) \in \wp(U)^\DOWN \times \wp(U)^\UP
\mid A^{\UP \vartriangle} \subseteq B \text{ and } A \cap \mathcal{S} = B \cap \mathcal{S}\},
\]
where $X^\vartriangle$ denotes the upper approximation of $X$ defined
in terms of the inverse $R^{-1}$ of the relation $R$, that is,
$X^{\vartriangle} = \{x\in U \mid R^{-1}(x) \cap X \neq \emptyset\}$.
If $R$ is a tolerance, we have $A^{\vartriangle} = A^{\UP}$ and
$A^{\UP\vartriangle} \subseteq B \iff A^{\UP\UP} \subseteq B 
\iff  A^{\UP} \subseteq B^{\DOWN}$ for any $A,B \subseteq U$. 
Additionally,
\begin{equation} \label{Eq:Congruence}
\mathcal{S} \subseteq A \cup B^{c} \iff \mathcal{S} \cap (B\setminus A) = \emptyset
\iff \mathcal{S} \cap B = \mathcal{S} \cap A. 
\end{equation}
Hence, for tolerances, we have 
$\mathcal{I}(\textit{RS})= \mathcal{DM}(\textit{RS})$, and consequently, 
$\mathit{RS} = \mathcal{I}(\mathit{RS})$ holds whenever $\mathit{RS}$ is a
complete lattice. In \cite{JPR12}, we proved that for any quasiorder $R$ on $U$, 
$\mathit{RS}$ is a complete, completely distributive lattice and
\[
\mathit{RS} = \{(A,B) \in \wp(U)^\DOWN \times \wp(U)^\UP \mid 
A \subseteq B \text{ and } \mathcal{S} \subseteq A \cup B^{c}\}.
\]
Therefore, $\mathcal{I}(\mathcal{RS})$ can be called the 
\emph{increasing representation of rough sets}.

A \textit{complete subdirect product} $\mathcal{L}$ of an indexed family of
complete lattices $\{L_i\}_{i \in I}$ is a complete sublattice of the direct product 
$\prod_{i \in I} L_i$ such that the canonical projections 
$\pi_i$ are all surjective, that is, $\pi_i(\mathcal{L}) = L_i$.
Note that the projections $\pi_i$ are complete lattice homomorphisms, 
that is, they preserve all meets and joins. 

\begin{proposition} \label{Prop:Completion}
Let $R$ be a tolerance on $U$.
\begin{enumerate}[\rm (a)]
\item  $\mathcal{I}(\mathit{RS})$ is a complete polarity sublattice of the polarity
lattice  $\wp(U)^\DOWN \times \wp(U)^\UP$.
\item $\mathcal{I}(\mathit{RS})$ is a complete subdirect product of  $\wp(U)^\DOWN$ and $\wp(U)^\UP$.
\end{enumerate}
\end{proposition}

\begin{proof}
(a) We first note that the map $\sim$ defined in \eqref{Eq:Negation} is a De~Morgan operation on 
$\mathcal{I}(\mathit{RS})$.
If $(A,B) \in \mathcal{I}(\mathit{RS})$, then $A^\UP \subseteq B^\DOWN$ implies
$B^{c \UP} = B^{\DOWN c} \subseteq A^{\UP c} = A^{c \DOWN}$. Additionally,
$\mathcal{S} \subseteq A \cup B^{c} = B^{c} \cup (A^c)^c$. So,
${\sim}(A,B) = (B^c,A^c) \in\mathcal{I}(\mathit{RS})$.

Let $\{(A_i,B_i)\}_{i \in I}\subseteq \mathcal{I}(\mathit{RS})$. 
Its meet defined in $\wp(U)^\DOWN \times \wp(U)^\UP$ is
\[ \bigwedge_{i \in I} (A_i,B_i)= 
\Big( \bigcap_{i \in I} A_i, \big ( \bigcap_{i \in I} B_{i} \big )^{\DOWN \UP} \Big) . \]
We show that this meet is in $\mathcal{I}(\mathit{RS})$. 
For all $i \in I$, we have ${A_i}^\UP \subseteq {B_i}^\DOWN$. 
Thus, for all $i \in I$, we have $A_i \subseteq {A_i}^{\UP \DOWN} \subseteq {B_i}^{\DOWN\DOWN}$, 
and $\bigcap_{i \in I} A_i \subseteq \bigcap_{i \in I} B_{i}^{\DOWN \DOWN} = 
(\bigcap_{i \in I} B_{i})^{\DOWN \DOWN}$. This implies
$(\bigcap_{i \in I} A_i)^\UP \subseteq (\bigcap_{i \in I} B_{i})^{\DOWN\DOWN \UP}
\subseteq (\bigcap_{i \in I} B_{i})^{\DOWN}
= \big((\bigcap_{i \in I} B_{i})^{\DOWN\UP}\big)^\DOWN$.
For the second part,
assume that $\mathcal{S} \nsubseteq \bigcap_{i \in I} A_i \cup 
(\bigcap_{i \in I} B_{i})^{\DOWN \UP c}$. This means that there
exists $x\in \mathcal{S}$ such that $x \notin \bigcap_{i \in I} A_i$
and $x \notin (\bigcap_{i \in I} B_i)^{\DOWN\UP c}$. Therefore,
there is $k \in I$ with $x \notin A_k$ and 
$x \in ( \bigcap_{i \in I} B_i )  ^{\DOWN\UP} \subseteq
\bigcap_{i \in I} B_i \subseteq B_k$. We get 
$x \notin A_{k} \cup {B_k}^{c}$, contradicting our assumption
$\mathcal{S} \subseteq A_i \cup {B_i}^c$ for all $i \in I$. 
Thus,  $\mathcal{S} \subseteq \bigcap_{i \in I} A_i \cup 
(\bigcap_{i \in I} B_{i})^{\DOWN \UP c}$ must hold.
Additionally, the join
\[ \bigvee_{i \in I} (A_i,B_i) = 
\Big (  \big (\bigcup_{i \in I} A_i \big )^{\UP \DOWN}, \bigcup_{i \in } B_i \Big )
\]
defined in $\wp(U)^\DOWN \times \wp(U)^\UP$ equals ${\sim} \bigwedge_{i\in I} {\sim}(A_i,B_i)$, which, 
by the above, belongs to $\mathcal{I}(\mathit{RS})$. Thus, $\mathcal{I}(\mathit{RS})$ is a 
complete sublattice of $\wp(U)^{\DOWN}\times\wp(U)^\UP$.

(b) The maps $\pi_{1} \colon (X^\DOWN, Y^\UP) \mapsto X^\DOWN$ and
$\pi_{2} \colon (X^\DOWN, Y^\UP) \mapsto Y^\UP$ are the canonical projections 
of the product  $\wp(U)^\DOWN \times \wp(U)^\UP$. Obviously,
their restrictions to $\mathcal{I}(\mathcal{RS})$ are surjective, because 
$\mathit{RS} \subseteq \mathcal{I}(\mathit{RS})$. Combined
with (a), this proves the claim.
\end{proof}

\begin{corollary} \label{Cor:Subdirect}
$\mathit{RS}$ is a complete lattice if and only if it is a complete subdirect 
product of the complete lattices $\wp(U)^\DOWN$ and $\wp(U)^\UP$.
\end{corollary}

\begin{remark}
The fact that for a tolerance $R$ on $U$, $\mathcal{I}(\mathit{RS})$ is the
Dedekind--MacNeille completion of $\mathit{RS}$ can be proved independently of 
\cite{UmaThesis} by showing that $\mathit{RS}$ is both join-dense and
meet-dense in $\mathcal{I}(\mathit{RS})$. It is known that a complete lattice $L$ is
the Dedekind--MacNeille completion of an ordered subset $P$ of $L$, whenever $P$ is
both join-dense and meet-dense in $L$, that is, every element of $L$ can be
represented as a join and a meet of some elements of $P$ 
(see e.g. \cite[Theorem~7.41]{DaPr02}). In fact, one can show that for any pair 
$(A,B)\in\mathcal{I}(\mathit{RS})$, we have
\begin{equation} \label{Eq:repre}
(A,B)= \bigvee \big ( \{ (R(x)^{\DOWN},R(x)^{\UP} \mid x\in A\} 
\cup \{ (\emptyset,R(x)) \mid x \in B^{\DOWN}\setminus A\} \big ).
\end{equation}
Trivially, the pairs $(R(x)^{\DOWN},R(x)^{\UP})$ are rough sets for every $x \in A$.
If $x \in B^\DOWN \setminus A$, then $x \notin \mathcal{S}$  by \eqref{Eq:Congruence}, and
hence $\{x\}^\DOWN = \emptyset$ and $(\{x\}^\DOWN,\{x\}^\UP) = (\emptyset,R(x))$ 
is a rough set. 
Thus, \eqref{Eq:repre} implies that $\mathit{RS}$ is join-dense in $\mathcal{I}(\mathit{RS})$. 
Because $\mathcal{I}(\mathit{RS})$ and $\mathit{RS}$ are self-dual by the map $\sim$, 
$\mathit{RS}$ is also meet-dense in $\mathcal{I}(\mathit{RS})$.
\end{remark}

By Corollary~\ref{Cor:Subdirect}, to show that  $\textit{RS}$ is a complete lattice
it is enough to prove that  $\textit{RS}$ is a complete sublattice of $\wp(U)^\DOWN \times \wp(U)^\UP$.
Additionally, since $\textit{RS}$ is a self-dual subset of the complete polarity lattice $\wp(U)^\DOWN \times \wp(U)^\UP$, 
it suffices to find for any $\mathcal{H} \subseteq \wp(U)$ a set $Z \subseteq U$ such that 
\begin{equation}\label{Eq:ExistZed}
Z^\DOWN =  \bigcap_{X \in \mathcal{H}} X^\DOWN \text{ \quad and \quad }  
Z^\UP =  \big (\bigcap_{X \in \mathcal{H}} X^\UP \big )^{\DOWN\UP}. 
\end{equation}
Observe that
\[  \big (\bigcap_{X \in \mathcal{H}} X \big )^{\DOWN\UP} =
\big (\bigcap_{X \in \mathcal{H}} X^\DOWN \big )^\UP = Z^{\DOWN\UP} \subseteq Z
 \subseteq Z^\UP =  \big (\bigcap_{X \in \mathcal{H}} X^\UP \big )^{\DOWN\UP}.
\]
So, we have a lower bound and an upper bound for this $Z$. Especially, concerning 
Lemma~\ref{Lem:ExistsS}, the interpretation is that
\[ T = \big (\bigcap_{X \in \mathcal{H}} X \big )^{\DOWN\UP} 
 \quad \mbox{and} \quad
   Y =  \big (\bigcap_{X \in \mathcal{H}} X^\UP \big )^{\DOWN\UP}.
\]

\begin{lemma} \label{Lem:ExistsS}
Let $Y,T\subseteq U$ be such that $Y \in \wp(U)^\UP$ and $T\subseteq Y^\DOWN$. 
If\/ $|R(x)| \geq2$ for all $x\in Y\setminus T^\UP$, then there exists a set 
$S\subseteq Y^\DOWN\setminus T$ such that $Y=S^\UP\cup T^\UP$  and
$R(y)\nsubseteq S\cup T$ for all $y\in S$.
\end{lemma}

\begin{proof}
Since $T\subseteq Y^{\DOWN} \subseteq Y$, the set $Y \setminus T$ is a disjoint union of 
$Y \setminus Y^{\DOWN}$ and $Y^{\DOWN} \setminus T$, and hence
\begin{equation} \label{Eq:DifferenceY}
(Y \setminus T)  \setminus (Y\setminus Y^{\DOWN}) = Y^{\DOWN} \setminus T.
\end{equation}
Clearly, $T^{\UP}\subseteq Y^{\DOWN\UP}=Y$, because $Y \in \wp(U)^{\UP}$. 
If $T^{\UP}=Y$, then our assertion is satisfied trivially with
$S = \emptyset$. Thus, we may suppose $T^{\UP}\subset Y$, which yields
$Y \setminus T \neq \emptyset$, because $T \subseteq$ $T^{\UP}$. Let
$\rho$ denote the restriction of $R$ to the set $Y \setminus T$. Then, $\rho$ is
a tolerance, and its transitive closure $\overline{\rho}$ is an equivalence on
$Y\setminus T$. Now, consider the sets
\begin{align*}
A & =\{ y \in Y\setminus T \mid \overline{\rho}(y) \cap (Y\setminus Y^{\DOWN})\neq\emptyset\}; \\
B &= \{y\in Y\setminus T\mid\overline{\rho}(y)\cap(Y\setminus
Y^{\DOWN})=\emptyset \text{ \ and \ } \overline{\rho}(y) \nsubseteq T^{\UP}\}.
\end{align*}
Then, $(Y\setminus Y^{\DOWN}) \cap B \subseteq A \cap B = \emptyset$, and by \eqref{Eq:DifferenceY} we obtain
\[
B \subseteq (Y\setminus T) \setminus (Y\setminus Y^{\DOWN}) = Y^{\DOWN}\setminus T.
\]

We apply Lemma~\ref{Lem:OddEven} with the tolerance $\rho$ and the sets 
$U = Y\setminus T$, $X = Y\setminus Y^{\DOWN}$, and 
$\overline{X} = \overline{\rho}(Y\setminus Y^{\DOWN}) = A$. 
We obtain two disjoint sets $(Y\setminus Y^{\DOWN})_{\text{odd}}$ 
and $(Y\setminus Y^{\DOWN})_{\text{even}}$ such that 
$(Y \setminus Y^{\DOWN})_{\text{odd}} \cup (Y\setminus Y^{\DOWN})_{\text{even}} = A$, and
\begin{align}
&Y\setminus Y^{\DOWN} \subseteq (Y\setminus Y^{\DOWN})_{\text{even}} \label{Eq:4A}; \\
&(Y\setminus Y^{\DOWN})_{\text{odd}}\subseteq \rho((Y\setminus Y^{\DOWN})_{\text{even}});  \label{Eq:4B} \\
&A \setminus (Y\setminus Y^{\DOWN}) \subseteq \rho((Y\setminus Y^{\DOWN})_{\text{odd}}).  \label{Eq:4C}
\end{align}

Next, let $\Pi=\{H_{k}\mid k\in K\}$ be the partition induced by the equivalence
$\overline{\rho}$ on $B$. Note that  $\overline{\rho}(B)=B$, 
because $x \in B$ and $\overline{\rho}(x) = \overline{\rho}(y)$ imply $y \in B$.
For each $k\in K$, we may select an element $c_{k}\in H_{k} \subseteq B$ 
such that $c_{k} \notin T^{\UP}$. This is because $H_{k}=\overline{\rho}(b)$ for some $b \in B$ 
and $B$ was defined so that $\overline{\rho}(b) \nsubseteq T^{\UP}$ for all $b\in B$. 
Denote the set of all these elements by $C$, that is,
\[
C=\{c_{k}\mid k\in K\}.
\]
Then, $C \subseteq B \setminus T^{\UP} \subseteq Y^{\DOWN} \setminus T^{\UP}$. 
Observe that $C \subseteq \rho(B\setminus C)$. Indeed, 
$C \subseteq Y \setminus T^{\UP}$ yields $|R(x)| \geq 2$ for all $x\in C$ by our
assumption. Thus, for each $x \in C$ there is an element $y \neq x$ with $x \, R \, y$.
Then, $y \in C^{\UP} \subseteq Y^{\DOWN\UP} \subseteq Y$, and $x \notin T^{\UP}$ yields 
$y\notin T$. Hence, $x,y \in Y \setminus T$ and $x \, \rho \, y$ holds also. 
This implies $y \in \overline{\rho}(x) \subseteq B$. Since $x \in C$ is the unique element 
picked from the set $\overline{\rho}(x) \in \Pi$, $y \neq x$ implies 
$y \in B\setminus C$. This proves $C \subseteq\rho(B\setminus C)$.

By applying  Lemma~\ref{Lem:OddEven} again with the tolerance $\rho$ and the sets 
$X = C \subseteq Y \setminus T$ and $\overline{X} = B$, we obtain two disjoint sets 
$C_{\text{odd}}$ and $C_{\text{even}}$ such that $C_{\text{odd}} \cup C_{\text{even}} = B$. 
Because $C \subseteq \rho(B\setminus C)$, we have
\begin{equation}\label{Eq:4D}
\rho(C_{\text{odd}})=\rho(C_{\text{even}})=B.
\end{equation}
Finally, consider the set
\[
S=(Y\setminus Y^{\DOWN})_{\text{odd}}\cup C_\text{odd}.
\]
We prove that $S$ has the required properties, that is, (i) $S\subseteq Y^{\DOWN}\setminus T$, 
(ii) $S^{\UP}\cup T^{\UP}=Y$, and (iii) $R(y) \nsubseteq S\cup T$ for all $y\in S$.

(i) Obviously, $C_\text{odd} \subseteq B \subseteq Y^{\DOWN} \setminus T$ and
$(Y\setminus Y^{\DOWN})_\text{odd} \subseteq A \subseteq Y\setminus T$.  
By \eqref{Eq:4A}, 
\[ (Y\setminus Y^{\DOWN})_\text{odd} \cap (Y\setminus Y^{\DOWN})
\subseteq (Y\setminus Y^{\DOWN})_\text{odd} \cap (Y\setminus Y^{\DOWN})_\text{even}=\emptyset. 
\]
Then, \eqref{Eq:DifferenceY} yields 
$(Y \setminus Y^{\DOWN})_\text{odd} \subseteq Y^{\DOWN} \setminus T$, and  
$S = (Y\setminus Y^{\DOWN})_{\text{odd}}\cup C_\text{odd} \subseteq Y^{\DOWN}\setminus T$.

(ii) Since $S,T\subseteq Y^{\DOWN}$, we have 
$S^{\UP} \cup T^{\UP} = (S\cup T)^{\UP}\subseteq Y^{\DOWN\UP}\subseteq Y$. 
For the other direction, assume $x \in Y\setminus T^{\UP} \subseteq Y \setminus T$. Then, 
either $\overline{\rho}(x) \cap (Y\setminus Y^{\DOWN}) =\emptyset$ or
$\overline{\rho}(x) \cap (Y\setminus Y^{\DOWN}) \neq \emptyset$ holds. 
Because $x \notin T^\UP$, we have $\overline{\rho}(x) \nsubseteq T^{\UP}$.
Therefore, if  $\overline{\rho}(x) \cap (Y\setminus Y^{\DOWN}) =\emptyset$, then
$x \in B$. Since $B = \rho(C_{\text{odd}})$ by \eqref{Eq:4D}
and $\rho(C_{\text{odd}})\subseteq S \subseteq S^{\UP}$, 
we get $x \in S^\UP$. If $\overline{\rho}(x) \cap (Y\setminus Y^{\DOWN}) \neq \emptyset$, then 
$x \in A$, and either $x \in Y \setminus Y^{\DOWN}$ or $x \notin Y\setminus Y^{\DOWN}$. 

If $x\in Y\setminus Y^{\DOWN}$, then $x\in Y = Y^{\DOWN\UP}$ yields that 
$x \, R \, y$ for some $y \in Y^{\DOWN}$. 
Since $x \notin T^{\UP}$, we have $y \notin T$ and $y \in Y \setminus T$. 
By $x,y\in Y \setminus T$, we obtain $x\, \rho \, y$. The facts
$x \in Y \setminus Y^{\DOWN}$ and $y \notin Y \setminus Y^{\DOWN}$ imply 
$ y \in(Y\setminus Y^{\DOWN})^{1} \subseteq (Y\setminus Y^{\DOWN})_{\text{odd}}$, 
which gives $x \in \rho((Y\setminus Y^{\DOWN })_{\text{odd}}) \subseteq 
((Y\setminus Y^{\DOWN})_{\text{odd}})^{\UP}\subseteq S^{\UP}$. 
If $x\notin Y\setminus Y^{\DOWN}$, then
\[ x \in A \setminus (Y \setminus Y^{\DOWN}) \subseteq \rho((Y\setminus Y^{\DOWN})_\text{odd})
\subseteq((Y\setminus Y^{\DOWN})_\text{odd})^{\UP} \subseteq S^{\UP}
\] 
by \eqref{Eq:4C}.
As we obtained $x \in S^{\UP}$ in all possible cases, 
$Y \subseteq S^{\UP}\cup T^{\UP}$ holds.

(iii) Let $x\in S = (Y\setminus Y^{\DOWN})_\text{odd} \cup C_\text{odd}$. 
Since $(Y\setminus Y^{\DOWN})_\text{odd} \cap C_\text{odd} \subseteq A \cap B = \emptyset$, 
either $x \in(Y \setminus Y^{\DOWN})_\text{odd}$ or $x\in C_\text{odd}$. 
In the first case, $x \in \rho((Y\setminus Y^{\DOWN})_\text{even})$ by \eqref{Eq:4B}. 
Hence, there is an element  $y \in (Y\setminus Y^{\DOWN})_\text{even}$ with 
$(x,y) \in \rho \subseteq R$, that is, $y\in R(x)$. 
The inclusion $(Y\setminus Y^{\DOWN})_\text{even} \subseteq  A \subseteq Y\setminus T$ gives 
$y \notin T$. Because 
\begin{align*}
S \cap (Y\setminus Y^{\DOWN})_\text{even} &= ((Y \setminus Y^{\DOWN})_\text{odd} 
\cap (Y\setminus Y^{\DOWN})_\text{even}) \cup (C_\text{odd} \cap (Y\setminus Y^{\DOWN})_\text{even})\\
& \subseteq \emptyset \cup(B\cap A) = \emptyset,
\end{align*} 
we get $y\notin S$. Thus, $R(x)\nsubseteq S\cup T$.
If $x\in C_{\text{odd}}$, then $x \in B = \rho(C_{\text{even}})$ by \eqref{Eq:4D}.
Hence, there is $y\in C_{\text{even}}$  with $(x,y) \in \rho \subseteq R$, that is, $y\in R(x)$. 
Since $C_\text{even} \subseteq B \subseteq Y \setminus T$, we get $y \notin T$. \ Clearly, 
$y \notin C_\text{odd}$ and 
$C_\text{even} \cap (Y\setminus Y^{\DOWN})_\text{odd} \subseteq B \cap A=\emptyset$
implies $y \notin(Y\setminus Y^{\DOWN})_\text{odd}$. Hence, 
$y \notin C_{\text{odd}} \cup(Y\setminus Y^{\DOWN})_\text{odd}=S$, 
and so $R(x) \nsubseteq S \cup T$.
\end{proof}

An element $x$ of a complete lattice $L$ is said to be \emph{compact} if for every subset
$S$ of $L$, $x \leq \bigvee S$ implies $x \leq \bigvee F$ for some finite subset $F$ of $S$.
A complete lattice is \emph{algebraic} if its every element can be given
as a join of compact elements.

\begin{theorem} \label{Thm:Dist2Lattice}
Let $R$ be a tolerance on $U$. Then $\mathit{RS}$ is an algebraic completely distributive
lattice if and only if $R$ is induced by an irredundant covering of $U$.
\end{theorem}

\begin{proof}
Suppose that $\mathit{RS}$ is an algebraic completely distributive lattice. 
Then, $\mathit{RS}$ is a complete lattice and, by Corollary~\ref{Cor:Subdirect}, 
it is a complete subdirect product of the  lattices $\wp(U)^\DOWN$ and $\wp(U)^\UP$. 
Since $\wp(U)^\DOWN$ is the image of $\mathit{RS}$ under the complete lattice-homomorphism $\pi_{1}$, 
the lattice $\wp(U)^\DOWN$ is also completely distributive. 
Hence, according to Proposition~\ref{Prop:ComplDistributive}, $R$ is induced
by an irredundant covering of $U$.

Conversely, let $R$ be a tolerance induced by an irredundant covering of $U$.
Then $R$ satisfies also condition (b) of Theorem~\ref{Thm:DC}. First, we show that 
$\textit{RS}$ is a complete lattice. Let $\mathcal{H} \subseteq \wp(U)$.
By \eqref{Eq:ExistZed}, it is enough to show that there exists a set 
$Z \subseteq U$ such that
\begin{equation}\label{Eq:MEET}
(Z^\DOWN, Z^\UP) =
\Big ( \bigcap_{X \in \mathcal{H}} X^\DOWN,  \big (\bigcap_{X \in \mathcal{H}} X^\UP \big )^{\DOWN\UP} \Big ).
\end{equation}
Let us first set 
\begin{equation}\label{Eq:DefT+Y}
 T = \big (\bigcap_{X \in \mathcal{H}} X \big )^{\DOWN\UP} 
 \quad \mbox{and} \quad
   Y =  \big (\bigcap_{X \in \mathcal{H}} X^\UP \big )^{\DOWN\UP}.
\end{equation}
Using the properties of $^\UP$ and $^\DOWN$, it is clear that  $T = T^{\DOWN\UP}$, $Y = Y^{\DOWN\UP}$, and
\[ T\subseteq \bigcap_{X \in \mathcal{H}} X \subseteq \bigcap_{X\in\mathcal{H}} X^{\UP\DOWN} = 
\Big ( \bigcap_{X \in \mathcal{H}} X^{\UP} \Big )^{\DOWN} = \Big (\bigcap_{X\in\mathcal{H}} X^{\UP} \Big ) ^{\DOWN\UP\DOWN} 
= Y^{\DOWN}.
\]
In addition, $T^\DOWN = \big (\bigcap_{X \in \mathcal{H}} X \big )^{\DOWN}$ and \
$T^\UP \subseteq Y^{\DOWN\UP} = Y$.
Suppose that $x \in Y \setminus T^\UP$ and $|R(x)| = 1$. Since
$R(x) = \{x\}$, $x \in Y$ implies that $x \in X$ for all $X \in \mathcal{H}$,
from which we get $x \in T^\UP$, a contradiction. Thus, we must have 
$|R(x)| \geq 2$ for all $x \in Y \setminus T^\UP$.
Now, we may apply Lemma~\ref{Lem:ExistsS} with the sets $T$ and $Y$, and this yields
that there exists a set 
\begin{equation}\label{Eq:DefS}
S \subseteq Y^\DOWN\setminus T \text{ with } S^\UP\cup T^\UP = Y
\text{ and } R(y) \nsubseteq S\cup T \text{ for all } y\in S. 
\end{equation}
Let us define the set
\[
V = \{v\in T\mid R(v) \nsubseteq T \mbox{ and } R(v) \subseteq S\cup T\}.
\]

First, we prove that if $V = \emptyset$, the set $Z = S \cup T$
satisfies \eqref{Eq:MEET}. Since $^\UP$ distributes over unions,
we have $Z^\UP = S^\UP \cup T^\UP = Y$. 
Trivially, $T^\DOWN \subseteq Z^\DOWN$. On the other hand, if $z \in Z^\DOWN$, 
then $R(z) \subseteq S \cup T$.
We have $R(z) \subseteq T$, because $R(z) \nsubseteq T$ 
implies $z \in V$, but this is impossible because $V = \emptyset$. Thus,
$Z^\DOWN = T^\DOWN$ and \eqref{Eq:MEET} holds.

\medskip

Now we prove that condition (b) of Theorem~\ref{Thm:DC} implies $V=\emptyset$, 
which by previous observation yields that $\mathit{RS}$ is a complete lattice.
Suppose $V \neq \emptyset$. Then, there exists
$v \in T$ such that $R(v) \nsubseteq T$ and 
$R(v) \subseteq S \cup T$. This also means that there is an element
$s \in R(v)$ with $s\in S\setminus T$. As $v \, R \, s$, by Theorem~\ref{Thm:DC}
there exists $c_{v} \in R(c_v) \subseteq R(v)\cap R(s)$ such that 
$R(c_{v})\subseteq R(x)$ for all $x \, R \, c_{v}$. 
Then $c_v \, R \,s$, and in particular, we have
\begin{equation} \label{EQ:INCL}
c_{v} \in R(c_v) \subseteq R(v)\subseteq S\cup T. 
\end{equation}
Observe that $c_{v} \notin T$, because $c_v \in T = T^{\DOWN\UP} $ means 
that there is  $a \in T^{\DOWN}$ with $a \, R \, c_{v}$. 
Since $R(c_{v})\subseteq R(x)$ for all $x \, R \, c_{v}$,
we get $s \in R(c_{v}) \subseteq R(a)\subseteq T^{\DOWN \UP} \subseteq T$, but
this is not possible because $s \in S\setminus T$. 
Therefore, $c_{v} \in S$, and we have $R(c_{v}) \nsubseteq S\cup T$ by
Lemma~\ref{Lem:ExistsS}. But this contradicts \eqref{EQ:INCL}, 
and we deduce $V = \emptyset$, which, as we have already noted, 
implies that $\mathit{RS}$ is a complete lattice.

Finally, since $\mathit{RS}$ is a complete lattice, it is isomorphic to a complete subdirect 
product of $\wp(U)^\DOWN$ and $\wp(U)^\UP$ by Corollary~\ref{Cor:Subdirect}.
Since $R$ is a tolerance induced by an irredundant covering, in view of Proposition~\ref{Prop:Boolean},
$\wp(U)^\DOWN$ and $\wp(U)^\UP$ are complete atomistic Boolean lattices. Thus, they are
completely distributive and algebraic, too. Hence, $\mathit{RS}$, as a complete subdirect product
of two completely distributive algebraic lattices is also completely distributive and algebraic.
\end{proof}

\begin{remark} \label{Rem:Bonikowski}
In \cite{Bonikowski1998}, Z.~Bonikowski with his co-authors considered rough set systems
of approximation pairs based on coverings. They presented a necessary and sufficient 
condition under which their system forms a complete lattice. This condition was given in terms of 
representative elements and representative coverings. To be more formal, let us recall
some notions from their work. Let $\mathcal{C} \subseteq \wp(U)$ be a covering. For any $x \in U$,
the family $\mathit{md}(x) = \min \{ K \in \mathcal{C} \mid x \in K \}$ is called the
\emph{minimal description of $x$}. In addition, for any $K \in \mathcal{C}$, 
the element $x \in K$ is called a \emph{representative element} of $K$ if
for all $S \in \mathcal{C}$, $x \in S$ implies $K \subseteq S$. The covering
$\mathcal{C}$ is called \emph{representative} if every set in $\mathcal{C}$ has
a representative element.

For any set $X \subseteq U$, the family $\mathcal{C}_*(X) = \{ K \in \mathcal{C} \mid K \subseteq X\}$ 
is called the \emph{sets bottom-approximating $X$}. The family 
$\mathit{Bn}(X) = \bigcup \{ \mathit{md}(x) \mid x \in X \setminus \bigcup \mathcal{C}_*(X)  \}$ 
is referred to the \emph{sets approximating the boundary of $X$}, and the family
$\mathcal{C}^*(X) = \mathcal{C}_*(X) \cup \mathit{Bn}(X)$ consists of the
\emph{sets top-approximating $X$}. It is proved in  \cite{Bonikowski1998} that
the coordinatewise ordered set $\{ (\mathcal{C}_*(X),\mathcal{C}^*(X)) \mid X \subseteq U\}$ is a complete lattice
if and only if
\begin{enumerate}[\rm (i)]
 \item the covering $\mathcal{C}$ is representative;
 \item every $K \in \mathcal{C}$ consisting at least two elements has at least two representative elements.
\end{enumerate}

Suppose that $R$ is a tolerance induced by an irredundant 
covering $\mathcal{C}$ of $U$. Then, by our Theorem~\ref{Thm:Dist2Lattice}, the ordered structure 
$\mathit{RS} = \{ (X^\DOWN,X^\UP) \mid X \subseteq U \}$ is always an algebraic completely
distributive lattice. On the other hand, the irreducible covering $\mathcal{C}$ is representable,
because by Proposition~\ref{Prop:IrredundantCovering}, for each $K \in \mathcal{C}$, there exists $d \in U$ 
such that $R(d)= B$. It is easy to observe that this $d$ is a representative element of $B$.
Thus, irredundant coverings always satisfy (i). But for an irredundant covering forming a lattice
in the sense of  \cite{Bonikowski1998}, also condition (ii) must be satisfied. Clearly, this means
that for every $K \in \mathcal{C}$ consisting at least two elements, there must be two distinct elements 
$d$ and $d'$ such that $R(d) = R(d') = K$. By this, it is now obvious that the systems $\mathit{RS}$ and
$\{ (\mathcal{C}_*(X),\mathcal{C}^*(X)) \mid X \subseteq U\}$ cannot be isomorphic in general.

Note also that in the case of an irredundant covering $\mathcal{C}$, the \emph{maximal description}
$\mathit{MD}(x) = \max \{ K \in \mathcal{C} \mid x \in K \}$ coincides with the minimal
description $\mathit{md}(x)$ of any $x \in U$ (cf. \cite{Restrepo2013,Yao2012}).
\end{remark}

An \emph{Alexandrov topology} is a topology $\mathcal{T}$ that
contains all arbitrary intersections of its members. Alexandrov
topologies are also called \emph{complete rings of sets}.
It is known that a lattice $L$ is  isomorphic to an Alexandrov topology
if and only if $L$ is completely distributive and algebraic (see e.g. \cite{DaPr02}). By
Theorem~\ref{Thm:Dist2Lattice}, $\mathit{RS}$ is isomorphic to some Alexandrov topology 
whenever $R$ is induced by an irredundant covering of $U$.

A \emph{Heyting algebra} $L$ is a bounded distributive lattice such that for all 
$a,b \in L$, there is a greatest element $x$ of $L$ such that $a \wedge x \leq b$.
This element is the \emph{relative pseudocomplement} of $a$ with respect to $b$, 
and is denoted $a \Rightarrow b$. It is well known that any completely distributive
lattice $L$ is a Heyting algebra such that the relative pseudocomplement is defined as 
\begin{equation} \label{Eq:Heyting}
x \Rightarrow y  =  \bigvee \big \{ z \in L \mid z \wedge x \leq y \big \}.
\end{equation}
Therefore, if $R$ is a tolerance induced by an irredundant covering of $U$, then 
$\mathit{RS}$ is a Heyting algebra.

A \textit{Kleene algebra} is a structure $\mathbb{A} = (A, \vee, \wedge, {\sim}, 0, 1)$ such that
$A$ is a bounded distributive lattice and for all $a,b \in A$:
\begin{enumerate}[({K}1)]
\item ${\sim}\,{\sim}a  =  a$,
\item $a \leq b  \text{ if and only if }  {\sim}b \leq {\sim}a$,
\item $a \wedge {\sim}a  \leq  b \vee {\sim}b$.
\end{enumerate}
According to R. Cignoli \cite{Cign86}, a \emph{quasi-Nelson algebra} is a Kleene algebra 
$\mathbb{A}$ such that for each pair $a$ and $b$ of its elements, the relative pseudocomplement
$a \Rightarrow ({\sim} a \vee b)$ exists. In quasi-Nelson algebras, 
$a \Rightarrow ({\sim} a \vee b)$ is denoted simply by
$a \to b$ and this is called the \emph{weak relative pseudocomplement} of $a$
with respect to $b$. Obviously, each Kleene algebra such that its
underlying lattice forms a Heyting algebra is a quasi-Nelson algebra.
A \emph{Nelson algebra} is a quasi-Nelson algebra 
$(A, \vee, \wedge, \to, {\sim}, 0, 1)$ satisfying the equation
\begin{equation*}\label{Eq:Nelson}
(a\wedge b)\rightarrow c = a \rightarrow (b\rightarrow c).
\end{equation*}

In the case of quasiorders, it is shown by J.~J{\"a}rvinen, S.~Radeleczki, and L.~Veres 
\cite{JRV09} that $\mathit{RS}$ forms a complete sublattice of $\wp(U) \times\wp(U)$ 
ordered by the coordinatewise set-inclusion relation.
In addition, we have proved in \cite{JarRad} that $\mathit{RS}$ determines a Nelson
algebra.

Let the operation $\sim$ on $\mathit{RS}$ be defined as in \eqref{Eq:Negation}.

\begin{proposition} \label{Prop:QuasiNelson}
Let $R$ be a tolerance induced by an irredundant covering of $U$. Then, the algebra
\[ \mathbb{RS} = (\mathit{RS},\cup,\cap,{\sim},(\emptyset,\emptyset), (U,U))\]
is a quasi-Nelson algebra.
\end{proposition}

\begin{proof}
If $R$ is a tolerance induced by an irredundant covering of $U$, then by Theorem~\ref{Thm:Dist2Lattice},
$\mathit{RS}$ is a complete distributive lattice bounded by $(\emptyset,\emptyset)$ and $(U,U)$. 
As we have already noted, conditions (K1) and (K2) are satisfied.
Let $\mathcal{A}(X) = (X^\DOWN,X^\UP)$ and $\mathcal{A}(Y) = (Y^\DOWN,Y^\UP)$ be in $\textit{RS}$. Then,
\begin{align*}
\mathcal{A}(X) \wedge {\sim} \mathcal{A}(X) &= (X^\DOWN \cap X^{c \DOWN} , (X^\UP \cap X^{c \UP})^{\DOWN \UP} ) 
                                         = (\emptyset, (X^\UP \setminus X^\DOWN)^{\DOWN \UP}), \mbox{ and} \\
\mathcal{A}(Y) \vee {\sim} \mathcal{A}(Y) &= ( (Y^\DOWN \cup Y^{c \DOWN})^{\UP \DOWN}), Y^\UP \cup  Y^{c \UP} ) 
= ((Y^{\DOWN\UP} \cup Y^{c \DOWN \UP})^\DOWN,U).
\end{align*}
Hence, $\mathcal{A}(X) \wedge {\sim} \mathcal{A}(X) \leq \mathcal{A}(Y) \vee {\sim} \mathcal{A}(Y)$, and condition (K3) holds also.

Since $\mathit{RS}$ is a Heyting algebra when $R$ is a tolerance induced by an irredundant covering of $U$,
the Kleene algebra $\mathbb{RS}$ is a quasi-Nelson algebra.
\end{proof}

Note that (K3) has nothing to do with distributivity, so if $\mathit{RS}$ is a lattice,
${\sim}$ satisfies conditions (K1)--(K3), and even $R$ is induced by an irredundant covering of $U$, 
the algebra $\mathbb{RS}$ does not necessarily form a Nelson algebra. For instance, if $R$ is a 
tolerance on $U = \{a,b,c\}$ such that $R(a) = \{a,b\}$, $R(b) = U$, and $R(c) = \{b,c\}$, the 
quasi-Nelson algebra $\mathbb{RS}$ is not a Nelson algebra.

\medskip%
For a tolerance $R$ on $U$ and any $X \subseteq U$, we denote the restriction of 
$R$ to $X$ by $R_X$ and by $\mathit{RS}_X$ the
set of all rough sets determined by the relation $R_X$ on $X$. It is clear that for
an equivalence $R$, the relation $R_X$ is an equivalence and $\mathit{RS}_X$
is a lattice for all $X \subseteq U$. Similar observation holds for quasiorders.

Let us introduce the following condition related to $R$-paths: 
\begin{itemize}
 \item[(C)] For any $R$-path $(a_0,\ldots,a_4)$ of length $4$, there
 exist $0\leq i,j \leq 4$ such that $|i - j| \geq 2$ and $a_i \, R \, a_j$.
 \end{itemize}

\begin{lemma} \label{Lem:Sufficient} 
 Let $R$ be a tolerance on $U$. If $\mathit{RS}_X$ is a lattice
 for all $X \subseteq U$ with $|X| = 5$, then $R$ satisfies condition\/ {\rm (C)}.
\end{lemma}

\begin{proof}
 Suppose $R$ does not satisfy (C). Then, there exists an $R$-path
 $(a_0, \ldots,a_4)$ such that $a_i \, R \, a_j$ if and only if
 $|i - j| \leq 1$. Let us choose $X = \{a_0, \ldots,a_4\}$. Then $|X| = 5$ and
 the situation is exactly as in Example~\ref{Ex:Counter}, that is, 
 $\mathit{RS}_X$ is not a lattice 
\end{proof}

Now, we present our second main result.

\begin{theorem} \label{Thm:Main}
If $R$ is a tolerance satisfying\/ {\rm (C)}, then $\mathit{RS}$ is a complete lattice.
\end{theorem}

\begin{proof}
Let $\mathcal{H} \subseteq \wp(A)$. As in the proof of Theorem~\ref{Thm:Dist2Lattice},
we need to find a set $Z \subseteq U$ such that $Z^\DOWN = \bigcap_{X \in \mathcal{H}} X^\DOWN$ 
and $Z^\UP = (\bigcap_{X \in \mathcal{H}} X^\UP )^{\DOWN\UP}$.
Let us form now the sets $T$, $Y$, $S$, and $V$ exactly as in the proof of Theorem~\ref{Thm:Dist2Lattice},
meaning that \eqref{Eq:DefT+Y} and \eqref{Eq:DefS} hold. Recall that
\[
V=\{v\in T\mid R(v)\nsubseteq T \text{ and } R(v) \subseteq S\cup T\}.
\]
According to the proof of Theorem~\ref{Thm:Dist2Lattice}, $V = \emptyset$ implies that
$\textit{RS}$ is a complete lattice, hence we may assume $V \neq \emptyset$.
Now, for each $v \in V$, we can  choose an element $q_{v} \in R(v)$ such that $q_{v} \in T^\UP\setminus T$.
Denote by $Q$ the set of these selected elements, that is, $Q = \{q_{v} \mid v\in V\}$. 
Then for all $v\in V$, $q_{v} \in T^\UP$, $q_v \notin T$, and $q_v \in R(v) \subseteq S \cup T$
give $q_{v}\in(S\setminus T)\cap T^\UP$, and so
\begin{equation}\label{Eq:Q}
Q \subseteq (S \setminus T) \cap T^\UP.
\end{equation}

Next, we define a set $P \subseteq Y$ by setting
\begin{equation}\label{Eq:defP}
P = \{p \in Y \setminus ((S\setminus Q)^{\UP} \cup T^{\UP}) \mid R(p)\subseteq Q^{\UP}\}.
\end{equation}
Then $P\subseteq Q^{\UP\DOWN} \subseteq Q^{\UP}$, because for each $p \in P$, 
$R(p)\subseteq Q^{\UP}$. In addition, 
\[ 
P\cap( (S \setminus Q)^\UP \cup T^\UP) = P \cap ((S\setminus Q) \cup T)^\UP = \emptyset, 
\]
that is,
\begin{equation}\label{Eq:ForAllP}
(\forall p \in P) \, R(p) \cap  \big ( (S\setminus Q) \cup T \big ) = \emptyset.
\end{equation}
We have $Q\subseteq T^\UP$ by \eqref{Eq:Q}, which gives
$P\cap Q \subseteq P \cap \big ( (S\setminus Q)^\UP \cup T^\UP \big ) = \emptyset$.

Let us now define the set
\[
Z = (S\setminus Q) \cup T\cup P.
\]
We will prove that 
\[
Z^\DOWN  =  \big ( \bigcap_{X \in \mathcal{H}} X \big )^\DOWN
= T^\DOWN 
\text{ \quad and \quad }
Z^\UP  =  \big (\bigcap_{X \in \mathcal{H}} X^\UP \big )^{\DOWN\UP} = Y.
\]

Trivially, $T^\DOWN \subseteq Z^\DOWN$.  To prove  $Z^\DOWN \subseteq T^\DOWN$, let $z\in Z^\DOWN$. 
Then,
\begin{equation}\label{Eq:Z}
z \in R(z) \subseteq Z = (S \setminus Q) \cup T \cup P,
\end{equation}
and we have $z \in S \setminus Q$, or $z \in T$, or $z \in P$. We first show that $z \in T$.

If $z\in P$, then $R(z)\cap \big ( (S\setminus Q)\cup T \big ) = \emptyset$ by \eqref{Eq:ForAllP},
and \eqref{Eq:Z} gives  $R(z)\subseteq P$. From this we obtain
$R(z) \cap Q \subseteq P\cap Q = \emptyset$. On the other hand, 
$z\in P\subseteq Q^\UP$ yields $R(z)\cap Q \neq \emptyset$, a contradiction.
Similarly, $z \in S\setminus Q$ gives $R(z) \subseteq (S\setminus Q)^\UP$. Then, 
$P \cap ((S \setminus Q)^\UP \cup T^\UP) = \emptyset$ implies 
$P \cap R(z)=\emptyset$ and we must have 
$R(z) \subseteq (S\setminus Q) \cup T \subseteq S \cup T$. 
Since $z \in S$, this contradicts $R(z) \nsubseteq S\cup T$ following from 
\eqref{Eq:DefS}. Hence, the only possibility left is $z \in T$. 

Next, we prove $z \in T^\DOWN$. Suppose, by the way of contradiction, that 
$R(z) \nsubseteq T$. Since $R(z)\subseteq T^\UP$ and, by \eqref{Eq:defP}, 
$P \cap T^\UP = \emptyset$, we obtain $R(z) \cap P \subseteq T^\UP \cap P = \emptyset$, 
which implies $R(z)\subseteq(S\setminus Q)\cup T\subseteq S\cup T$. 
This means that $z\in V$. So, there exists an element $q_{z}\in Q$
with $q_{z}\in R(z)$. By the definition of $Q$, $q_{z}\notin T$. So, 
we have $q_{z}\notin(S\setminus Q)\cup T$, which contradicts 
$R(z)\subseteq (S\setminus Q)\cup T$. Therefore, we have now proved 
$R(z) \subseteq T$, that is, $z\in T^\DOWN$. 

To complete our proof, we need to show that $Z^\UP=Y$. Recall that 
$Y = S^\UP\cup T^\UP$ by \eqref{Eq:DefS}.
By the definition of $Z$, we have $Z^\UP=(S\setminus Q)^\UP\cup T^\UP\cup
P^\UP$. In view of \eqref{Eq:defP}, $R(p) \subseteq Q^\UP$ for all $p \in P$, which
implies $P^\UP = \bigcup_{p \in P} R(p) \subseteq Q^\UP \subseteq S^\UP$, 
because $Q \subseteq (S \setminus T) \cap T^\UP \subseteq S$ holds by \eqref{Eq:Q}.
Hence, we have $Z^\UP\subseteq(S\setminus Q)^\UP\cup T^\UP\cup
S^\UP=T^\UP\cup S^\UP$. We show $Z^\UP=S^\UP\cup T^\UP$ by  
proving $(S^\UP\cup T^\UP) \setminus Z^\UP = \emptyset$.

Assume now that $R$ satisfies (C) and suppose for contradiction that there exists an 
element $y \in(S^\UP\cup T^\UP)\setminus Z^\UP = Y \setminus Z^\UP$. 
Since $Q \subseteq S$, we have $S = Q \cup (S\setminus Q)$ and 
$y\in Q^\UP\cup(S\setminus Q)^\UP\cup T^\UP$.
Because $y \notin Z^\UP = (S\setminus Q)^\UP \cup T^\UP\cup P^\UP$
yields $y \notin (S\setminus Q)^\UP$ and $y \notin T^\UP$, 
we must have $y\in Q^\UP \setminus T^\UP$. 
As $Q\subseteq T^\UP$, we get $y \in Q^\UP\setminus Q$.
Therefore, there are $v\in V\subseteq T$ and $q_{v}\in Q$ 
such that $q_{v}\in R(v)\subseteq S\cup T$, $R(v)\nsubseteq T$,
and $y\in R(q_{v})$. Note that since $y \notin Q$, we have $y\neq q_{v}$. 
Because $q_{v}\notin T$, we obtain $v\neq q_{v}$ also. So, there exist
$v,q_v,y$ such that $v \neq q_v$, $q_v \neq y$, $v \, R \, q_v$,
and $q_v \, R \, y$.

Because $v\in T=T^{\DOWN\UP}$, there is $a \in T^\DOWN$ such 
that $a \, R \, v$. The fact that $R(v) \nsubseteq T$ gives $v \notin T^\DOWN$ and 
hence we must have $a\neq v$. 
Observe also that $R(y)\subseteq Q^\UP$ is not possible. This is because
$y \notin Z^\UP = (S \setminus Q)^\UP \cup T^\UP \cup P^\UP$,
that is, $y \notin (S \setminus Q)^\UP$, $y \notin T^\UP$,
and $y \notin P^\UP$, combined with $y \in S^\UP \cup T^\UP = Y$, 
yield $y \in Y \setminus ((S\setminus Q)^\UP \cup T^\UP)$. 
Hence, by \eqref{Eq:defP}, $R(y)\subseteq Q^\UP$ would imply 
$y\in P\subseteq P^{\UP}$, a contradiction. Therefore,
$R(y) \nsubseteq Q^\UP = \bigcup_{q\in Q} R(q)$, and so there is an element $u\in R(y)$ such that
\begin{equation} \label{Eq:ForAllQ}
(\forall q \in Q) \, u \notin R(q).
\end{equation}
Then $y \, R \,u$, and clearly $u\neq y$, because $y \, R \, q_v$ holds.

We need to prove there are no $R$-related elements in the $R$-path
$(a,v,q_{v},y,u)$ except two consecutive ones. If this is true, then all the elements 
of the path are distinct, because $a \neq v$, $v \neq q_{v}$, $q_{v},\neq y$,
$y\neq u$ and $R$ is reflexive. Since this is a contradiction to our assumption that $R$ satisfies (C), 
there is no $y\in(S^\UP \cup T^\UP) \setminus Z^\UP$, and we may conclude that $Z^\UP = S^\UP \cup T^\UP=Y$,
which finishes the proof.

Indeed, $a \, R \, q_{v}$ implies $q_{v} \in R(a)\subseteq T^{\DOWN \UP} = T$, contradicting 
$q_{v}\notin T$. Similarly, $a \, R \, y$ and $v \, R \, y$
are not possible, because $a,v \in T$ and $y \notin T^\UP$. By \eqref{Eq:ForAllQ}, 
$q_{v} \, R \, u$ cannot hold. Furthermore, $a \, R \,u$ implies 
$u \in R(a) \subseteq T$ and $y \in R(u) \subseteq T^\UP$, a contradiction.
Finally, since $R(v) \subseteq S\cup T$, $v \, R \, u$ implies 
$u \in R(v)\subseteq S\cup T$. Moreover, we get $u \in(S\setminus Q)\cup T$, 
because \eqref{Eq:ForAllQ} implies $u \notin Q$. So, this yields 
$y \in R(u) \subseteq(S \setminus Q)^{\UP}\cup T^{\UP} \subseteq Z^{\UP}$, a
contradiction again. Hence, neither $v \, R \, u$ is possible.
\end{proof}

\begin{lemma}\label{Lem:ConditionC}
Any tolerance $R$ on $U$ satisfies {\rm (C)} if and only if for any $X\subseteq U$,
${R_X}^{3}$ is an equivalence on $X$.
\end{lemma}

\begin{proof}
The relation ${R_X}^{3}$ is a tolerance on any $X \subseteq U$ and ${R_X}^{3} \subseteq {R_X}^{4}$. 
If $(x,y) \in {R_X}^{4}$, then there is
an $R$-path $(a_{0},...,a_{4})$ of length $4$ with $a_{0} = x$ and $a_{4} = y$.
Condition (C) implies that there are $0 \leq i,j \leq 4$ such that $|i-j| \geq 2$
and $a_i \, R \, a_j$. Hence,  $(x,y) \in {R_X}^{3}$ and ${R_X}^{4} = {R_X}^{3}$. 
Additionally, we can see by induction that (C) implies ${R_X}^{n} = {R_X}^{3}$ for all $n \geq 3$. Then,
${R_X}^{3}\circ {R_X}^{3} = {R_X}^{6} = {R_X}^{3}$ and hence the tolerance
${R_X}^{3}$ is transitive, that is, ${R_X}^{3}$ is an equivalence. Conversely, let $(a_{0},...,a_{4})$ be an
$R$-path of length $4$, $X = \{a_{0},...,a_{4}\}$, and suppose that ${R_X}^{3}$
is an equivalence. Then $(a_{0},a_{4}) \in {R_X}^{4} \subseteq {R_X}^{6}= 
{R_X}^{3} \circ {R_X}^{3} \subseteq {R_X}^{3}$ implies 
$(a_{0},a_{4}) \in {R_X}^{3}$. Observe that this is possible only if condition (C) holds.
\end{proof}

\begin{corollary} \label{Cor:Condition}
Let $R$ be a tolerance on $U$. Then, the following are equivalent:
\begin{enumerate}[\rm (a)]
 \item $\mathit{RS}_{X}$ is a complete lattice for all $X\subseteq U$.

 \item $\mathit{RS}_{X}$ is a lattice for all $X\subseteq U$ with $|X| = 5$.

 \item For any $X \subseteq U$, ${R_X}^{3}$ is an equivalence on $X$.
\end{enumerate}
\end{corollary}

\begin{proof}
The implication (a)$\Rightarrow$(b) is trivial. If (b) holds, then $R$ satisfies condition (C) according to 
Lemma~\ref{Lem:Sufficient}. Hence, by Lemma~\ref{Lem:ConditionC}, 
every ${R_X}^{3}$ is an equivalence, and we have (b)$\Rightarrow$(c). Again,
by Lemma~\ref{Lem:ConditionC}, (c) implies that $R_{X}$ satisfies (C) for all $X\subseteq U$. 
Hence, by applying  Theorem~\ref{Thm:Main} for each $X\subseteq U$ and $R_X$, we obtain (a), and 
so (c)$\Rightarrow$(a).
\end{proof}

\begin{example}
Let $\mathcal{S} = (U,A,\{V_a\}_{a \in A})$ be an information in which each attribute is two-valued,
that is, $V_a = \{0,1\}$ for all $a \in A$. For any $B \subseteq A$, the \emph{weak $B$-indiscernibility} is 
defined so that for all $x,y \in U$,
\[ 
(x,y) \in \mathit{wind}_B \iff (\exists a \in B)\, a(x) = a(y). 
\]
Let $B \subseteq A$ and assume that there is a $\mathit{wind}_B$-path $(x_1,x_2,x_3,x_4,x_5)$ in $U$.
This means that for each $1 \leq i \leq 4$, there is an attribute $a \in B$ such that $a(x_i) = a(x_{i+1})$.

Assume that condition (C) does not hold. Then, in particular, $(x_1,x_3) \notin  \mathit{wind}_B$, $(x_3,x_5) \notin \mathit{wind}_B$, 
and $(x_1,x_5) \notin  \mathit{wind}_B$. 
This means that for all $a \in B$, $a(x_1) \neq a(x_3)$ and $a(x_3) \neq a(x_5)$. But since the attribute sets are two-valued,
this must imply that $a(x_1) = a(x_5)$ for all $a \in B$. Thus, $(x_1,x_5) \in  \mathit{wind}_B$, a contradiction.

The following information system shows that (C) does not necessarily hold in cases when attribute sets have at least three values.
\begin{table}[h]
\centering
\begin{center}
\begin{tabular}{c|cc}
$U$ & $a$ & $b$ \\ \hline
$1$ & $0$ & $0$\\
$2$ & $0$ & $1$\\
$3$ & $1$ & $1$\\
$4$ & $1$ & $2$\\
$5$ & $2$ & $2$
\end{tabular}
\end{center}
\end{table}
\end{example}

\section{Disjoint representation of rough sets} \label{Sec:DisjointRepresentations}

Disjoint representations of rough sets were introduced by P.~Pagliani in \cite{Pagliani97}.
Each rough set $(X^\DOWN,X^\UP)$ may as well be represented as a
pair $(X^\DOWN,X^{\UP c})$, called the \emph{disjoint rough set} of $X$.
Clearly, $(X^\DOWN,X^{\UP c}) \in \wp(U)^\DOWN \times \wp(U)^\DOWN$
and now $X^{\UP c}$ can be interpreted as the set of elements that certainly are outside $X$,
while $X^\DOWN$ consists of elements certainly belonging to $X$. Let us denote
\[ \mathit{dRS} = \{ (X^\DOWN, X^{\UP c}) \mid X \subseteq U \}, \]
and define an order-isomorphism $\phi$ between $\wp(U)^\DOWN \times \wp(U)^\UP$ and
$\wp(U)^\DOWN \times \wp(U)^{\DOWN \mathrm{op}}$ by $(A,B) \mapsto (A,B^c)$. Obviously, $\phi$
is also an order-isomorphism between $\mathit{RS}$ and $\mathit{dRS}$, when
$\mathit{dRS}$ is ordered by the order of $\wp(U)^\DOWN \times \wp(U)^{\DOWN \mathrm{op}}$.
We define a De~Morgan operation $\mathfrak{c}$ on $\wp(U)^\DOWN \times \wp(U)^{\DOWN \mathrm{op}}$ by
\begin{equation}\label{Eq:Swap}
\mathfrak{c} \colon (A,B) \to (B,A).
\end{equation}
Clearly, for all $(A,B) \in \wp(U)^\DOWN \times \wp(U)^\UP$,
\[ \phi({\sim}(A,B)) = \phi(B^c,A^c) = (B^c,A) = \mathfrak{c}(A,B^c) = \mathfrak{c}(\phi(A,B)),\]
where ${\sim}$ is the De~Morgan operation on $\wp(U)^\DOWN \times \wp(U)^\UP$ defined in \eqref{Eq:Negation}.
Additionally, if $(X^\DOWN,X^{\UP c}) \in \mathit{dRS}$, then  $\mathfrak{c}(X^\DOWN,X^{\UP c})  = 
(X^{c \DOWN},X^{c \UP c}) \in \mathit{dRS}$.

In \cite{Pagliani97} Pagliani showed that in the case of equivalences, disjoint rough sets are closely 
connected to the construction of Nelson algebras by Sendlewski \cite{Sendlewski90}.  
Pagliani's  results are generalized for quasiorders in \cite{JPR12}, where it is proved that for any 
quasiorder $R$ on $U$,
\[ 
 \mathit{dRS} = \{ (A,B) \in \wp(U)^\DOWN \times \wp(U)^\DOWN \mid A \cap B = \emptyset \text{ and }
   \mathcal{S} \subseteq A \cup B \}, 
\]
where $\mathcal{S}$ is the set of singleton $R(x)$-sets defined as in \eqref{EQ:DefineS}. 
By applying this equality it is possible to show that on $\mathit{dRS}$, and thus on $\mathit{RS}$, 
a Nelson algebra can be defined by applying Sendlewski's construction. However, in the case of tolerances 
the situation is quite different, because $\mathit{RS}$ and $\mathit{dRS}$ do not always form lattices, and
even they do, the lattices are not necessarily distributive. However, in case
the tolerance $R$ induced by an irredundant covering of $U$, these lattices are distributive, and a 
quasi-Nelson algebra can be defined on $\mathit{RS}$ and $\mathit{dRS}$, as shown in
Proposition~\ref{Prop:QuasiNelson}. Anyway, these quasi-Nelson algebras are not necessarily
Nelson algebras.

In Section~\ref{Sec:OrderedSets}, we defined the increasing representation of
rough sets, that is,
\[
\mathcal{I}(\mathit{RS}) = \{ (A,B) \in \wp(U)^{\DOWN} \times \wp(U)^{\UP}
\mid A^{\UP}\subseteq B^\DOWN \text{ and } \mathcal{S} \subseteq A\cup B^{c}\},
\]
and showed that this is the Dedekind--MacNeille completion of $\mathit{RS}$.
If we map the set $\mathcal{I}(\mathit{RS})$ by the isomorphism $\phi$, we obtain the set
\[
\mathcal{D}(\mathit{RS}) = \{ (A,B) \in \wp(U)^{\DOWN} \times \wp(U)^{\DOWN}
\mid A^{\UP} \cap B^\UP = \emptyset \text{ and } \mathcal{S} \subseteq A\cup B\}.
\]
The set $\mathcal{D}(\mathit{RS})$ is called the \emph{disjoint representation of rough sets}. Obviously,
the map $\mathfrak{c}$ defined in \eqref{Eq:Swap} is a De~Morgan operation
on $\mathcal{D}(\mathit{RS})$,  and if $\mathit{RS}$ is a complete lattice, 
then $\mathit{RS}$ and $\mathcal{D}(\mathit{RS})$ can be identified by the 
map $\phi$.

We end this work by studying the connection between $\mathcal{D}(\mathit{RS})$ 
and the concept lattice $\mathfrak{B}(\mathbb{K})$ defined by the context 
$\mathbb{K} = (U,U,R^c)$. In \cite{Kwuida04,Wille2000}, it is considered for a  
concept $(A,B)$ of an arbitrary context its \emph{weak negation} by 
\[
(A, B)^\bigtriangleup  = (A^{c \prime \prime}, A^{c \prime}) 
\]
and its \textit{weak opposition} by
\[
(A,B)^\bigtriangledown = (B^{c \prime}, B^{c \prime\prime}).
\]
Especially, we are here considering the weak opposition operation $^\bigtriangledown$, 
which satisfies for all concepts $(A,B)$ and $(C,D)$:
\[
 (A,B) \leq (C,D)^{\bigtriangledown} \iff (C,D) \leq (A,B)^{\bigtriangledown}.
\]
We already noted in Section~\ref{Sec:ToleranceApproximations} that the concept lattice 
of the context $\mathbb{K} = (U,U,R^c)$ is 
$\mathfrak{B}(\mathbb{K}) = \{ (A,A^\top) \mid A \in \wp(U)^\DOWN \}$,
where $^\top$ is the orthocomplement operation of $\wp(U)^\DOWN$.
Recall that $A^\top = A^{\prime} = A^{\UP c} = A^{c \DOWN}$.
For $(A,B) \in \mathfrak{B}(\mathbb{K})$, the weak negation and the weak opposition are then
defined by 
\[ 
(A, B)^\bigtriangleup = (A^{\DOWN \top}, A^{\DOWN})
\text{ \ and \ }
(A, B)^\bigtriangledown  =  (B^{\DOWN}, B^{\DOWN \top}).
\]
We now consider the complete lattice
${\mathfrak{B}}(\mathbb{K}) \times {\mathfrak{B}}(\mathbb{K})^\mathrm{op}$,
where ${\mathfrak{B}}(\mathbb{K})^\mathrm{op}$ is the dual of the concept lattice
${\mathfrak{B}}(\mathbb{K})$, that is, ${\mathfrak{B}}(\mathbb{K})^\mathrm{op}$ is ordered by 
\[
 (A_1,B_1) \leq (A_2,B_2) \iff A_1 \supseteq A_2 \iff B_1 \subseteq B_2.
\]
Let $\mathrm{ext}(\alpha)$ denote the extent $A$ of a concept $\alpha =(A,B)$. We
define the set
\[
\mathcal{FC}(\mathit{RS}) = \{ (\alpha,\beta) \in \mathfrak{B}(\mathbb{K}) \times \mathfrak{B}(\mathbb{K})
\mid \beta \leq \alpha^{\bigtriangledown} \text{ in $\mathfrak{B}(\mathbb{K})$ and } \mathcal{S} \subseteq 
\mathrm{ext}(\alpha) \cup \mathrm{ext}(\beta) \},
\]
and we call it the \emph{formal concept representation of rough sets}.
We order the set $\mathcal{FC}(\mathit{RS})$ by the order of 
${\mathfrak{B}}(\mathbb{K}) \times {\mathfrak{B}}(\mathbb{K})^\mathrm{op}$.
 
\begin{proposition}
Let $R$ be a tolerance on $U$ and \ $\mathbb{K} = (U,U,R^c)$. 
\begin{enumerate}[\rm (a)]
 \item $\mathcal{FC}(\mathit{RS})$ is a complete sublattice of 
 $\mathfrak{B}(\mathbb{K}) \times \mathfrak{B}(\mathbb{K})^\mathrm{op}$.
 \item The complete lattices $\mathcal{FC}(\mathit{RS})$ and
 $\mathcal{D}(\mathit{RS})$ are isomorphic.
\end{enumerate}
\end{proposition}

\begin{proof} First, let us define a map $\varphi \colon \wp(U)^\DOWN \times \wp(U)^\DOWN 
\to \mathfrak{B}(\mathbb{K}) \times  \mathfrak{B}(\mathbb{K})$
by setting
\[(A,B) \mapsto ((A,A^\top),(B,B^\top)).
\]
Trivially, the map $\varphi$ is well defined. Next we show that $\varphi$
is an order-isomorphism between the complete lattices $\wp(U)^\DOWN \times \wp(U)^{\DOWN\mathrm{op}}$
and $\mathfrak{B}(\mathbb{K}) \times  \mathfrak{B}(\mathbb{K})^\mathrm{op}$.

If $(A,B),(C,D)\in \wp(U)^\DOWN \times \wp(U)^\DOWN$, then
\begin{align*}
& (A,B) \leq (C,D) \text{ in $\wp(U)^\DOWN \times \wp(U)^{\DOWN\mathrm{op}}$} &\iff \\ 
& A \subseteq C \text{ and } B \supseteq D & \iff \\
& (A,A^\top) \leq (C,C^\top) \text{ and } (B,B^\top) \geq (D,D^\top) \text{ in $\mathfrak{B}(\mathbb{K})$} &\iff \\ 
& \varphi(A,B)\leq \varphi(C,D) \text{ in $\mathfrak{B}(\mathbb{K}) \times \mathfrak{B}(\mathbb{K})^\mathrm{op}$}
\end{align*}
Thus, $\varphi$ is an order-embedding. If $((A,A^\top), (B,B^\top)) \in 
\mathfrak{B}(\mathbb{K}) \times \mathfrak{B}(\mathbb{K})$, then
$A,B \in \wp(U)^\DOWN$ and $\varphi(A,B) = ((A,A^\top), (B,B^\top))$. Therefore, the 
map $\varphi$ is also onto, and it is an order-isomorphism.

Next, we prove that $\mathcal{FC}(\mathit{RS})$ is the image
of $\mathcal{D}(\mathit{RS})$ under $\varphi$. Note first that for all
$A,B \in \wp(U)^\DOWN$,
\begin{align*}
 A^{\UP} \cap B^\UP = \emptyset  \iff A^{\UP \UP} \cap B = \emptyset 
  \iff B \subseteq A^{\UP \UP c} = A^{c \DOWN \DOWN} = A^{\top \DOWN}.
\end{align*}
Since $(A,A^\top)^\bigtriangledown = (A^{\top \DOWN}, A^{\top \DOWN \top})$,
we have that 
\[ A^{\UP} \cap B^\UP = \emptyset \iff (B,B^\top) \leq (A,A^\top)^\bigtriangledown  \text{ in $\mathfrak{B}(\mathbb{K})$}.\]
Additionally, $\mathcal{S} \subseteq A \cup B = \mathrm{ext}(A,A^\top) \cup \mathrm{ext}(B,B^\top)$.
These facts imply 
\[
(A,B) \in\mathcal{D}(\mathit{RS}) \iff \varphi(A,B) \in\mathcal{FC}(\mathit{RS}).
\] 
Since $\varphi$ is a bijection, we get 
\[
\mathcal{FC}(\mathit{RS}) = (\varphi \circ \varphi^{-1})(\mathcal{FC}(\mathit{RS})) = \varphi(\mathcal{D}(\mathit{RS})).
\]
Hence, $\varphi$ determines an order-isomorphism between the complete lattices $\mathcal{D}(\mathit{RS})$
and $\mathcal{FC}(\mathit{RS})$, which proves (b).

Because $\mathcal{I}(\mathit{RS})$ is a complete sublattice of $\wp(U)^\DOWN \times \wp(U)^\UP$
by Proposition~\ref{Prop:Completion}(b), its image
$\mathcal{D}(\mathit{RS})$ under the isomorphism $\phi \colon (A,B) \mapsto (A,B^c)$ 
is a complete sublattice of $\wp(U)^\DOWN \times \wp(U)^{\DOWN \mathrm{op}}$.
This implies that the image $\mathcal{FC}(\mathit{RS})$ of $\mathcal{D}(\mathit{RS})$ under $\varphi$
is a complete sublattice of 
${\mathfrak{B}}(\mathbb{K})\times {\mathfrak{B}}(\mathbb{K})^\mathrm{op}$, and claim (a)
is proved. 
\end{proof}

Since for all $(\alpha,\beta) \in \mathcal{FC}(\mathit{RS})$,
$\beta \leq \alpha^{\bigtriangledown} \iff \alpha \leq \beta^{\bigtriangledown}$,
it is easy to see that $\mathfrak{c}^* \colon (\alpha,\beta) \mapsto (\beta,\alpha)$ 
is a De~Morgan operation on $\mathcal{FC}(\mathit{RS})$. Up to isomorphism,
this operation is the same as in $\mathcal{I}(\mathit{RS})$ and
$\mathcal{D}(\mathit{RS})$. Namely, if $(A,B) \in \mathcal{D}(\mathit{RS})$, then
\[
\varphi( \mathfrak{c}(A,B)) = \varphi(B,A) = ((B,B^\top),(A,A^\top)) =
\mathfrak{c}^*((A,A^\top),(B,B^\top)) =  \mathfrak{c}^*(\varphi(A,B) ). \]

\medskip\medskip\noindent%
We conclude this section by giving the following summary of rough representations:
\begin{itemize}
 \item For any tolerance $R$ on $U$, the representations $\mathcal{I}(\mathit{RS})$,
 $\mathcal{D}(\mathit{RS})$, and $\mathcal{FC}(\mathit{RS})$ are Dedekind--MacNeille
 completions of $\mathit{RS}$ equipped with De~Morgan operations satisfying (K3) 
 that are identical up to isomorphism.
\item The ordered sets $\mathit{RS}$ and $\mathit{dRS}$ are isomorphic, and they are
 complete lattices if and only if $\mathit{RS}$ is a complete sublattice of
 $\wp(U)^\DOWN \times \wp(U)^\UP$, or, equivalently, $\mathit{dRS}$ is a complete sublattice of
 $\wp(U)^\DOWN \times \wp(U)^{\DOWN \textrm{op}}$. 
\item If  $\mathit{RS}$ and $\mathit{dRS}$ are complete lattices, then
 they are identical to $\mathcal{I}(\mathit{RS})$ and $\mathcal{D}(\mathit{RS})$, respectively.
 This implies also $\mathit{RS} \cong \mathcal{FC}(\mathit{RS})$.
\item If $R$ induced by an irredundant covering of $U$, then $\mathit{RS}$, $\mathit{dRS}$,  $\mathcal{I}(\mathit{RS})$,
 $\mathcal{D}(\mathit{RS})$, and $\mathcal{FC}(\mathit{RS})$ determine isomorphic quasi-Nelson algebras. 
\end{itemize}

\section{Some concluding remarks}

In this work, we have considered rough set systems determined by so-called element-based approximation pairs induced by a tolerance relation. 
For any tolerance, we were able to give the Dedekind--MacNeille completion of  $\mathit{RS}$  in terms of formal concept analysis, and also in the 
terms of increasing and, respectively, disjoint representations. Under some certain conditions, the rough set system   $\mathit{RS}$  forms a complete 
lattices. Particularly, if the tolerance is induced by an irredundant covering of the universe, its rough set lattice is algebraic and completely 
distribute, and a Kleene algebra (in fact, even a quasi-Nelson algebra) may be defined on it.

We learned that weak similarity satisfies condition (C) in the case attributes are two-valued, but this is generally no longer true even
for three-valued attributes. Additionally, we observed that tolerances induced by irredundant 
coverings arise in incomplete and approximate information systems in the presence of learning examples. 
We also would like to emphasize that if we have a finite universe $U$ and we want that the rough set lattice $\mathit{RS}$  is distributive, 
this means that necessarily $\mathit{RS}$  
is determined by a tolerance induced by an irredundant covering. This is because if a finite lattice is distributive, it is 
also completely distributive. Since quite often in studies of rough set theory it is assumed that the universe is finite, this means 
that case of $\mathit{RS}$  forming a distributive lattice is completely characterized.

In the future, we will study under which conditions rough sets systems determined by tolerances define Nelson algebras or 3-valued {\L}ukasiewicz algebras, 
because it is known that in case of quasiorders the rough set systems form Nelson algebras, and in the particular case of equivalences, 
these systems define 3-valued {\L}ukasiewicz algebras. We propose a deeper analysis of the tolerance relations induced by irredundant coverings, 
and their relations to information systems. 
It might also be fruitful to study lattice-theoretical properties of completions considered in Section~\ref{Sec:DisjointRepresentations},
that is, how the properties of the ortholattices $\wp(U)^\DOWN \cong \wp(U)^\UP$ effect to the completions. These constructions may
have some similarities to the one of Sendlewski \cite{Sendlewski90} or they could be based on generalizations of
Heyting algebras \cite{Chajda03}, for instance.
Finally, we note that it still remains an open question under which condition on the tolerance $R$, the rough set system forms a lattice.

\section*{Acknowledgements}
The authors would like to thank the referees and  Yiyu Yao for their constructive comments.


\end{document}